\documentclass[onefignum,onetabnum]{siamart220329}

\makeatletter
\AtBeginDocument{%
  \def\refstepcounter@optarg[#1]#2{%
    \cref@old@refstepcounter{#2}%
    \cref@constructprefix{#2}{\cref@result}%
    \@ifundefined{cref@#1@alias}%
      {\def\@tempa{#1}}%
      {\def\@tempa{\csname cref@#1@alias\endcsname}}%
    \protected@edef\cref@currentlabel{%
      [\@tempa][\arabic{#2}][\cref@result]%
      \csname p@#2\endcsname\csname the#2\endcsname}%
  }%
}
\makeatother

\makeatletter
\@twosidefalse
\makeatother

\usepackage{algorithm,algorithmic}
\usepackage{amsmath,amscd,amsopn,amssymb,bm,mathrsfs,mathtools}
\usepackage[a4paper]{geometry}
\usepackage{graphicx}
\usepackage{lipsum}
\usepackage{hyperref}
\usepackage{setspace}
\usepackage{subfigure}

\graphicspath{{fig/}}
\DeclareGraphicsExtensions{.pdf}

\newsiamremark{remark}{Remark}
\newsiamthm{assumption}{Assumption}

\hypersetup{colorlinks=true,linkcolor=black,citecolor=siaminlinkcolor,urlcolor=siamexlinkcolor}

\DeclareMathOperator{\diag}{diag}

\DeclareMathOperator{\fl}{f{}l}

\DeclareMathOperator{\VEC}{vec}

\newcommand*{\R}{\mathbb{R}}
\newcommand*{\C}{\mathbb{C}}
\newcommand*{\F}{\mathbb{F}}

\newcommand*{\abs}[1]{\lvert#1\rvert}
\newcommand*{\eiggamma}{\gamma_{\mathsf{eig}}}
\newcommand*{\fro}{\mathsf F}
\newcommand*{\herm}{^{*}}

\newcommand*{\itrans}{^{-\top}}
\newcommand*{\machepso}{\bm{u}}
\newcommand*{\machepsl}{\bm{u}_{l}}

\newcommand*{\norm}[1]{\lVert#1\rVert_{2}}
\newcommand*{\bignorm}[1]{\bigl\lVert#1\bigr\rVert_{2}}
\newcommand*{\Bignorm}[1]{\Bigl\lVert#1\Bigr\rVert_{2}}
\newcommand*{\normfro}[1]{\lVert#1\rVert_{\fro}}

\newcommand*{\Bignormfro}[1]{\Bigl\lVert#1\Bigr\rVert_{\fro}}

\newcommand*{\trans}{^{\top}}

\headers{A mixed precision algorithm for the matrix square root}{B. Gao, D. Kressner, and M. Shao}

\title{A mixed precision algorithm for the matrix square root%
\thanks{Bowen Gao and Meiyue Shao are partially supported by National
Key R\&D Program of China under Grant No.\ 2023YFB3001603.
Bowen Gao completed part of the work during his visit in EPFL.}}

\author{
Bowen Gao\thanks{School of Data Science, Fudan University, Shanghai 200433,
China (\email{bwgao22@m.fudan.edu.cn}, \email{myshao@fudan.edu.cn})} \and
Daniel Kressner\thanks{Institute of Mathematics, EPF Lausanne,
1015 Lausanne, Switzerland (\email{daniel.kressner@epfl.ch})} \and
Meiyue Shao\footnotemark[2]~\thanks{Shanghai Key Laboratory for Contemporary Applied Mathematics, Fudan University, Shanghai 200433, China}}

\begin{document}

\maketitle

\begin{abstract}
Mixed precision algorithms can significantly enhance the performance
of linear algebra solvers by leveraging increasingly powerful
low precision hardware while recovering working precision
accuracy through, for example, iterative refinement.
In this paper, we propose a novel mixed precision
algorithm for computing matrix square roots.
Our algorithm combines a Schur decomposition approach in low precision
with iterative refinement performed through an approximate Newton method.
We perform a detailed convergence analysis of the approximate Newton method.
For the special case of symmetric positive definite matrices,
this analysis implies that one can recover full
working precision accuracy under mild conditions.
Numerical experiments on x86-64 architectures indicate
that our algorithm frequently reduces execution time
compared with a fixed working-precision Schur algorithm.
\end{abstract}

\begin{keywords}
Matrix square root, matrix function, mixed precision computation,
iterative refinement, Newton method
\end{keywords}

\begin{MSCcodes}
65F60, 65F45, 64G50
\end{MSCcodes}

\section{Introduction}
\label{sec:introduction}
Matrix functions represent an important subject in numerical
linear algebra, and their computation has long sparked great
interest within the scientific computing community,
with applications in control theory, physics, and engineering.
A common and crucial matrix function is the matrix
square root.
Given a matrix \(A\in\C^{n\times n}\) with no eigenvalues on the non-positive
real axis \(\R_{\leq 0}\), the \emph{(principal) matrix square root} \(A^{1/2}\)
of \(A\) is the unique matrix \(X \in\C^{n\times n}\) such that \(X^{2}=A\) 
and all eigenvalues of \(X\) are in the open right-half complex plane. 

Direct methods for computing matrix square roots are usually
derived from the Schur--Parlett algorithm, often simply referred to
as the Schur algorithm; see~\cite[Algorithms 6.3 and 6.5]{Higham2008}.
It starts by computing a Schur decomposition \(A=QTQ\herm\),
where \(T\) is an upper triangular matrix, and then uses
a back-substitution procedure to compute \(T^{1/2}\) from \(T\).
The square root is obtained via \(A^{1/2}=QT^{1/2}Q\herm\).
The Schur algorithm is the most widely used method for
computing matrix square roots, owing to its favorable numerical
stability properties~\cite[Section~6.2]{Higham2008}.

Matrix iterations such as 
\begin{equation}
\label{eq:Newton}
X_{k+1}=\frac{1}{2}(X_{k}+AX_{k}^{-1}),\qquad X_{0}=A
\end{equation}
represent another popular way to compute matrix square roots;
see~\cite[Section 6.3]{Higham2008}.
Iteration~\eqref{eq:Newton} is obtained from applying the Newton method
to the matrix equation \(X^{2}-A=0\) and using the commutativity of
the iterates \(X_{k}\) with \(A\) to simplify the correction equation.
For an overview and analysis of matrix iterations for computing matrix roots,
we refer to~\cite{BHM2005,Higham2008} and the references therein.

Mixed precision computation bears a long
and rich history in numerical linear algebra.
Traditional mixed precision algorithms, such as classical iterative
refinement~\cite{Moler1967}, were originally introduced
to enhance the accuracy of potentially numerically unstable algorithms.
In contrast, modern mixed precision algorithms are primarily performance driven.
They strategically employ lower precision computation to reduce execution time,
and use higher precision computation to recover the lost accuracy.
A typical way is to use lower precision computation to attain an approximate
solution, and then carry out iterative refinement in working/higher precision.

Mixed precision techniques have by now been applied to a broad
range of problems in numerical linear algebra~\cite{Survey2021}.
For example, algorithms for solving linear systems have been discussed
in~\cite{CH2017,CH2018,LLLKBD2006}, while algorithms for linear least squares
problems and their variants are covered in~\cite{CHP2020,GMS2025-2,OC2024}.
A modular framework for the backward error
analysis of GMRES is presented in~\cite{BHMV2025}.
For eigenvalue and singular value problems, iterative refinement techniques
have been developed in~\cite{GMS2025-1,KMS2023,OA2018,OA2019,OA2020}
for the symmetric case and in~\cite{BKS2023} for the nonsymmetric case.
Algorithms for computing matrix functions in arbitrary precision
have been investigated in~\cite{Fasi2019,Liu2022}.
In the context of matrix functions, mixed precision has so far
mainly been used to reduce the cost for attaining high accuracy and
increase numerical robustness; see~\cite{HL2021,Liu2025} for examples.

In this work, we will develop and analyse novel mixed precision
algorithms for computing the matrix square root \(X=A^{1/2}\).
The general idea is to first compute an approximation
\(X_{0}\) in lower precision, and then refine it via
iterative refinement to attain the working precision accuracy.
It is important to note that~\eqref{eq:Newton}
is \emph{not} suited for this purpose.
Due to the violation of the commutativity assumption used
in deriving this iteration, the error of this method
usually stagnates on the level of the lower precision.
This limitation can be overcome by going back to
the (true) Newton method applied to \(X^{2}-A=0\).
This comes at a cost; the correction equation becomes a matrix
Sylvester equation, which needs to be solved in every step.
We will mitigate this by freezing the Jacobian (leading to an approximate
Newton method), using lower precision for solving the correction equation,
and reusing the Schur decomposition from computing \(X_{0}\).

The remainder of this paper is structured as follows.
In Section~\ref{sec:algorithm}, we outline the development and later
propose mixed precision algorithms for computing matrix square roots.
In Section~\ref{sec:analysis}, we perform a detailed analysis
to demonstrate the convergence and accuracy of the proposed methods.
We extend our mixed precision computation framework
to matrix \(p\)th roots in Section~\ref{sec:pthroot}.
Numerical experiments are provided in Section~\ref{sec:experiments} to
demonstrate the accuracy and efficiency of our mixed precision algorithms.

In this paper, we let \(\norm{A}\) and \(\normfro{A}\) denote the \(2\)-norm
and the Frobenius norm of matrix~\(A\), respectively,
and \(\kappa_{2}(A)=\norm{A}\norm{A^{-1}}\).
Let \(\otimes\) denote the Kronecker product and \(\VEC(\cdot)\) denote
the operator that transforms a matrix into a vector by stacking its columns.
\section{Schur-based mixed precision algorithms for matrix square roots}
\label{sec:algorithm}
In this section, we derive our newly proposed mixed precision
algorithms for computing the matrix square root \(X=A^{1/2}\).

The design of our refinement scheme is guided by a key observation
from mixed precision iterative refinement for linear systems:
the attainable accuracy is largely determined by the precision
used for residual computation and carrying out corrections~\cite{CH2018}.
Carrying over this principle to matrix square roots, our approach uses
lower precision arithmetic for computing an initial square root approximation
as well as for solving the correction equation in each refinement step.
Higher (working) precision arithmetic is reserved for computing
the residual and adding the correction to improve the approximation.
Given a square root approximation \(\hat{X}\), the corrected matrix
\(\hat{X}+\Delta X\) ideally satisfies \((\hat{X}+\Delta X)^{2}=A\).
As we expect \(\Delta X\) to be small, we neglect the second power
of \(\Delta X\), arriving at the following correction equation:
\begin{equation}
\label{eq:correction}
\hat{X}\cdot\Delta X+\Delta X\cdot\hat{X}=R,
\end{equation}
which is a so called matrix Sylvester equation.

The considerations above lead to the following framework:
\begin{enumerate}
\item  Compute an initial approximation \(X_{0}\) of the matrix
square root of \(A\) in lower precision and set \(k\gets 0\);
\item  Compute the residual \(R_{k}=A-X_{k}^{2}\) in working precision;
\item  Compute \(\Delta X_{k}\) by solving the correction equation
\(X_{k}\cdot\Delta X_{k}+\Delta X_{k}\cdot X_{k}=R_{k}\) in lower precision.
\item  Add \(X_{k+1}\gets X_{k}+\Delta X_{k}\)
in working precision and update \(k\gets k+1\).

Repeat steps 2--4 until the residual \(\normfro{R_{k}}\) is sufficiently small.
\end{enumerate}

In the rest of this section, we will discuss the implementation of Step 3,
the solve of~\eqref{eq:correction}.

\subsection{Reusing Schur decompositions}
\label{subsec:reuse}
Several numerical methods for solving Sylvester equations
of the form~\eqref{eq:correction} have been developed;
see~\cite{Simoncini2016} and the references therein.
It turns out that a variant of the classical Bartels--Stewart
method~\cite{BS1972} is very well suited for our purpose;
it is general and numerically robust, and more crucially, it allows us
to conveniently reuse the Schur decomposition computed for \(X_{0}\).

We use the Schur algorithm~\cite[Section 6.2]{Higham2008}
carried out in lower precision to compute \(X_{0}\).
As discussed in Section~\ref{sec:introduction},
it starts with computing a Schur decomposition \(A=QTQ\herm\),
where \(T\) is upper triangular and \(Q\) is unitary.
The square root \(T^{1/2}\) is then also upper triangular and is computed
using the blocked method described in Section~\ref{subsubsec:blockhigham} below.
Finally, \(X_{0}\) is obtained from the matrix product \(X_{0}=QT^{1/2}Q\herm\).

A common technique in the Newton method is to
freeze the Jacobian throughout the iteration.
In the context of our framework above,
it means that we freeze \(\hat{X}=X_{0}\) in the correction
equation~\eqref{eq:correction} used in Step 3 of our framework as
\[
X_{0}\cdot\Delta X_{k}+\Delta X_{k}\cdot X_{0}=R_{k}.
\]
Reusing the Schur decomposition \(X_{0}=QT^{1/2}Q\herm\)
allows us to rewrite this equation as
\begin{equation}
\label{eq:triangsylv}
S\cdot Y+Y\cdot S={R},
\end{equation}
with \(S=T^{1/2}\), \(\Delta X_{k}=QYQ\herm\),
and the updated right-hand side \(R\gets Q\herm R_{k}Q\).
Because \(S\) has all its eigenvalues in the open right-half complex plane,
\eqref{eq:triangsylv} is uniquely solvable.
The triangular structure of \(S\) allows one to solve~\eqref{eq:triangsylv}
in \(\mathcal{O}(n^{3})\) operations using a back-substitution procedure.
In Section~\ref{subsubsec:blocksylv}, we will discuss how this
can be carried out efficiently with a blocked algorithm.

\subsection{Blocked algorithms}
\label{subsec:block}
To enhance the performance of our algorithm, we use blocked algorithms
to compute the square root of a triangular matrix and solve upper
triangular Sylvester equations of the form~\eqref{eq:triangsylv}.

\subsubsection{Blocked algorithm for computing
square roots of upper triangular matrices}
\label{subsubsec:blockhigham}
The standard method for computing the matrix square root
\(S=T^{1/2}\) of an upper triangular matrix \(T\) is the following
back-substitution technique~\cite[Algorithm 6.3]{Higham2008}:

\begin{enumerate}
\item  Compute the diagonal entries \(s_{i,i}=t_{i,i}^{1/2}\)
for \(i=1,\dotsc,n\);
\item  Compute the other nonzero entries by column.
For \(j=2:n\), \(i=j-1:-1:1\), update
\begin{equation}
\label{eq:subst}
s_{i,j}=\frac{t_{i,j}-\sum_{k=i+1}^{j-1}s_{i,k}s_{k,j}}{s_{i,i}+s_{j,j}}.
\end{equation}
\end{enumerate}

In~\cite{DHR2013}, Deadman, Higham, and Ralha proposed
different blocked algorithms for this procedure.
In particular, the block method partitions
\(T\) in blocks of roughly equal size.
It then carries out the back-substitution technique above on these blocks.
For the diagonal blocks, square roots of (small) upper
triangular matrices are computed using the technique above.
The off-diagonal blocks are computed in (block) column order,
solving (small) Sylvester equations instead of~\eqref{eq:subst}.
Additionally, \cite{DHR2013} proposed a block recursion method,
which recursively partitions \(T\) into two-by-two block matrices,
but our preliminary numerical experiments reveal little
to no performance advantage in our computational setting.

\begin{remark}
When matrix \(A\) is real, a real Schur decomposition \(A=QTQ\trans\) is
used, where \(Q\) becomes a real orthogonal matrix and \(T\) is quasi-upper
triangular, with \(1\times1\) and \(2\times2\) blocks on the diagonal.
The real Schur decomposition is cheaper than the complex Schur decomposition,
and only minor changes need to be made to the implementation thereafter.
In particular, when performing a block partition,
no \(2\times2\) blocks are cut.
To simplify the description, we will focus on using
the complex Schur form in the following sections.
\end{remark}

\subsubsection{Block recursion for solving triangular Sylvester equations}
\label{subsubsec:blocksylv}
To solve the triangular Sylvester equation~\eqref{eq:triangsylv}
in step 3 of our framework, we will make use of the block recursion
technique by Jonsson and K{\aa}gstr{\"o}m~\cite{JK2002-1,JK2002-2}.
Applied to~\eqref{eq:triangsylv}, this technique starts by
partitioning the involved matrices as follows (where \(\tilde{S}=S\)):
\begin{equation}
\label{eq:partition}
S=\begin{bmatrix}S_{1,1} & S_{1,2} \\ 0 & S_{2,2}\end{bmatrix},\quad
\tilde{S}=\begin{bmatrix}\tilde{S}_{1,1} & \tilde{S}_{1,2} \\
0 & \tilde{S}_{2,2}\end{bmatrix},\quad
Y=\begin{bmatrix}Y_{1,1} & Y_{1,2} \\ Y_{2,1} & Y_{2,2}\end{bmatrix},\quad
R=\begin{bmatrix} R_{1,1} & R_{1,2} \\ R_{2,1} & R_{2,2}\end{bmatrix},
\end{equation}
such that the diagonal blocks are of size roughly \(n/2\).
Then \(SY+Y\tilde{S}=R\) becomes
\begin{align*}
S_{1,1}Y_{1,1}+Y_{1,1}\tilde{S}_{1,1} & =R_{1,1}-S_{1,2}Y_{2,1},\\
S_{1,1}Y_{1,2}+Y_{1,2}\tilde{S}_{2,2} &
=R_{1,2}-S_{1,2}Y_{2,2}-Y_{1,1}\tilde{S}_{1,2},\\
S_{2,2}Y_{2,1}+Y_{2,1}\tilde{S}_{1,1} & =R_{2,1},\\
S_{2,2}Y_{2,2}+Y_{2,2}\tilde{S}_{2,2} & =R_{2,2}-Y_{2,1}\tilde{S}_{1,2}.
\end{align*}

This reduces the original equation to \(4\) smaller triangular
Sylvester equations, which are solved recursively until the matrix sizes
are sufficiently small, when the standard Bartels--Stewart method is used.
The resulting procedure is outlined in Algorithm~\ref{alg:recursion}.
Numerical experiments~\cite{JK2003} show that the block recursion
algorithm offers significant speedup compared with applying
the standard Bartels--Stewart method to the original Sylvester equation.

\begin{algorithm}[!tb]
\caption{Block recursion algorithm for solving
(quasi-)upper triangular Sylvester equations}
\label{alg:recursion}\footnotesize
\begin{algorithmic}[1]
\REQUIRE  Matrices \(S\in\F^{m\times m}\), \(\tilde{S}\in\F^{n\times n}\)
in (quasi-)upper triangular Schur form, matrix \(R\in\F^{m\times n}\),
where \(\F\in\{\R,\C\}\), and minimal block size \texttt{blks}
that specifies when to switch to a standard solver.
\ENSURE  Matrix \(Y\in\F^{m\times n}\) satisfying \(SY+Y\tilde{S}=R\).\\
\medskip
\hspace{-2\algorithmicindent}
\textbf{Interface}: \(Y=\mathtt{rtrsyl}(S,\tilde{S},R,\mathtt{blks})\)

\medskip
\IF{\(1\leq m,n\leq\mathtt{blks}\)}
\STATE  Solve \(Y\) directly with the Bartels--Stewart method.
\ELSIF{\(1\leq n\leq m/2\)}
\STATE  Partition \(S\) as in~\eqref{eq:partition},
and \(C\herm=[C_{1}\herm~C_{2}\herm]\) by rows only.
\STATE  \(Y_{2}=\mathtt{rtrsyl}(S_{2,2},\tilde{S},C_{2},\texttt{blks})\)
\STATE  \(C_{1}\gets C_{1}-S_{1,2}Y_{2}\)
\STATE  \(Y_{1}=\mathtt{rtrsyl}(S_{1,1},\tilde{S},C_{1},\texttt{blks})\)
\STATE  \(Y\herm=[Y_{1}\herm~Y_{2}\herm]\)
\ELSIF{\(1\leq m\leq n/2\)}
\STATE  Partition \(\tilde{S}\) as in~\eqref{eq:partition},
and \(C=[C_{1}~C_{2}]\) by columns only.
\STATE  \(Y_{1}=\mathtt{rtrsyl}(S,\tilde{S}_{1,1},C_{1},\texttt{blks})\)
\STATE  \(C_{2}\gets C_{2}-Y_{1}\tilde{S}_{1,2}\)
\STATE  \(Y_{2}=\mathtt{rtrsyl}(S,\tilde{S}_{2,2},C_{2},\texttt{blks})\)
\STATE  \(Y=[Y_{1}~Y_{2}]\)
\ELSE
\STATE  Partition \(S\), \(\tilde{S}\), and \(C\)
by rows and columns as in~\eqref{eq:partition}.
\STATE  \(Y_{2,1}=\mathtt{rtrsyl}(S_{2,2},
\tilde{S}_{1,1},C_{2,1},\texttt{blks})\)
\STATE  \(C_{1,1}\gets C_{1,1}-S_{1,2}Y_{2,1}\),
\(C_{2,2}\gets C_{2,2}-Y_{2,1}\tilde{S}_{1,2}\)
\STATE  \(Y_{1,1}=\mathtt{rtrsyl}(S_{1,1},
\tilde{S}_{1,1},C_{1,1},\texttt{blks})\)
\STATE  \(Y_{2,2}=\mathtt{rtrsyl}(S_{2,2},
\tilde{S}_{2,2},C_{2,2},\texttt{blks})\)
\STATE  \(C_{1,2}\gets C_{1,2}-S_{1,2}Y_{2,2}-Y_{1,1}\tilde{S}_{1,2}\)
\STATE  \(Y_{1,2}=\mathtt{rtrsyl}(S_{1,1},
\tilde{S}_{2,2},C_{1,2},\texttt{blks})\)
\STATE  \(\setlength{\arraycolsep}{3pt}
Y=\begin{bmatrix}Y_{1,1} & Y_{1,2}\\ Y_{2,1} & Y_{2,2}\end{bmatrix}\)
\ENDIF
\RETURN  \(Y\)
\end{algorithmic}
\end{algorithm}

\begin{remark}
Even though~\eqref{eq:triangsylv} is a particular Sylvester equation
with \(\tilde{S}=S\), Algorithm~\ref{alg:recursion} is presented
for solving a general triangular Sylvester equation because in lower levels
of recursion the Sylvester equations are no longer in that particular form.
\end{remark}

\subsection{Main algorithms}
\label{subsec:alg-descriptions}
Combining the techniques from Sections~\ref{subsec:reuse}
and~\ref{subsec:block} with our framework leads to Algorithm~\ref{alg:main},
one of the main algorithms proposed in this work.

\begin{algorithm}[!htb]
\caption{Mixed precision algorithm for computing matrix square roots}
\label{alg:main} \footnotesize
\begin{algorithmic}[1]
\REQUIRE  Matrix \(A\in\F^{n\times n}\) with no eigenvalue on \(\R_{\leq 0}\),
where \(\F\in\{\R,\C\}\), maximal number of iterations \texttt{maxit},
convergence threshold \texttt{tol}, and minimal block size \texttt{blks}.
\ENSURE  Matrix \(X\in\F^{n\times n}\) being the matrix square root of \(A\).
\medskip
\STATE  Compute Schur decomposition \(A=QTQ\herm\) in lower precision.
\STATE  Compute matrix square root \(S\) of \(T\)
using the blocked algorithm in lower precision.
\STATE  Compute \(X_{0}=QSQ\herm\) in lower precision.
\FOR{\(k=0,\dotsc,\mathtt{maxit}\)}
\STATE  Compute residual \(R=A-X_{k}^{2}\) in working precision.
\IF{\(\normfro{R}\leq\mathtt{tol}\cdot\normfro{A}\)}
\RETURN  \(X=X_{k}\).
\ENDIF
\STATE  Update \(R\leftarrow Q\herm RQ\) in lower precision.
\STATE  Solve (block) upper triangular Sylvester equation
\(SY+YS=R\) with Algorithm~\ref{alg:recursion} in lower precision:
\[
Y=\mathtt{rtrsyl}(S,S,R,\mathtt{blks}).
\]
\vspace{-\baselineskip}
\STATE  Compute \(\Delta X_{k}=QYQ\herm\) in lower precision.
\STATE  Add \(X_{k+1}=X_{k}+\Delta X_{k}\) in working precision.
\ENDFOR
\end{algorithmic}
\end{algorithm}

A special case of particular practical importance is the computation of
matrix square roots of symmetric, or Hermitian, positive definite matrices.
In this setting, the Schur decomposition reduces to the spectral decomposition.
Additionally, operations such as computing the square root of
a triangular matrix and solving triangular Sylvester equations
reduce to computing scalar square roots and solving diagonal
Sylvester equations, both of which are straightforward.
The resulting simplified procedure is summarized in Algorithm~\ref{alg:spd}.

\begin{algorithm}[!htb] 
\caption{Mixed precision algorithm for computing matrix square
roots of symmetric / Hermitian positive definite matrices}
\label{alg:spd}\footnotesize
\begin{algorithmic}[1]
\REQUIRE  Real symmetric or complex Hermitian positive definite matrix
\(A\in\F^{n\times n}\), where \(\F\in\{\R,\C\}\), maximal number
of iterations \texttt{maxit}, and convergence threshold \texttt{tol}.
\ENSURE  Matrix \(X\in\F^{n\times n}\) being the matrix square root of \(A\).
\medskip
\STATE  Compute spectral decomposition \(A=QTQ\herm\) in lower precision.
\STATE  Compute diagonal matrix \(S=T^{1/2}\)
by setting \(s_{i,i}=t_{i,i}^{1/2}\), \(i=1,\dotsc,n\).
\STATE  Compute \(X_{0}=QSQ\herm\) in lower precision.
\FOR{\(k=0,\dotsc,\mathtt{maxit}\)}
\STATE  Compute residual \(R=A-X_{k}^{2}\) in working precision.
\IF{\(\normfro{R}\leq\mathtt{tol}\cdot\normfro{A}\)}
\RETURN  \(X=X_{k}\).
\ENDIF
\STATE  Update \(R\leftarrow Q\herm RQ\) in lower precision.
\STATE  Solve diagonal Sylvester equation \(SY+YS=R\)
in lower precision by setting
\[
y_{i,j}=\frac{r_{i,j}}{s_{i,i}+s_{j,j}},\qquad i,j=1,\dotsc,n.
\]
\vspace{-0.5\baselineskip}
\STATE  Compute \(\Delta X_{k}=QYQ\herm\) in lower precision.
\STATE  Add \(X_{k+1}=X_{k}+\Delta X_{k}\) in working precision.
\ENDFOR
\end{algorithmic}
\end{algorithm}
\section{Convergence and accuracy analysis}
\label{sec:analysis}
In this section, we perform a convergence and accuracy analysis
of our mixed precision algorithms for computing matrix square roots.
For simplicity, we focus our analysis on real matrices
as the analysis for complex matrices is very similar.
In Section~\ref{subsec:convergence},
we analyse convergence in exact arithmetic and prove that
the iterative refinement converges at least linearly to the matrix square
root if the Schur algorithm provides a sufficiently close approximation.
In Sections~\ref{subsec:accuracy-spd} and~\ref{subsec:accuracy-general},
we analyse the effect of rounding errors.
This turns out to be significantly simpler for the symmetric
positive definite case, for which we establish forward stability of
Algorithm~\ref{alg:spd} under some mild conditions.
Parts of the analysis carry over to
Algorithm~\ref{alg:main} for general matrices.

In this section, we use \(\fl(\cdot)\) and \(\fl_{l}(\cdot)\) to denote
computed values in working and lower precision, respectively.
Letting \(\machepso\) and \(\machepsl\) denote the corresponding unit
roundoffs, we follow the notation from~\cite{Higham2002} and set 
\begin{equation}
\label{eq:gamma}
\gamma_{n}=\frac{n\machepso}{1-n\machepso},\qquad
\gamma_{n}^{l}=\frac{n\machepsl}{1-n\machepsl}.
\end{equation}
Throughout this section, it is tacitly assumed
that all these quantities are well defined.
For example, when writing \(\gamma_{n}^{l}\),
it is assumed that \(n\machepsl<1\).

\subsection{Convergence analysis}
\label{subsec:convergence}
To establish convergence of the iterative refinement process
in Algorithms~\ref{alg:main} and~\ref{alg:spd} to
the matrix square root \(X=A^{1/2}\), we derive relations
between the iteration errors \emph{in exact arithmetic}.
It is assumed that the iteration is started by some \emph{approximation}
\(X_{0}\), e.g., the one returned by the Schur algorithm in lower precision.
Let \(B_{X_{0}}=I_{n}\otimes X_{0}+X_{0}\trans\otimes I_{n}\).
Then each iteration in Algorithms~\ref{alg:main}
and~\ref{alg:spd} can be compactly written as
\begin{equation}
\label{eq:vec}
\VEC(X_{k+1})=\VEC(X_{k})+B_{X_{0}}^{-1}\cdot\VEC(A-X_{k}^{2}).
\end{equation}
By the Newton--Kantorovich theorem~\cite[Section 12.6.4]{OR2000},
this iteration converges if
\[
\norm{B_{X_{0}}^{-1}}\cdot\bignorm{B_{X_{0}}^{-1}
\cdot\VEC(A-X_{0}^{2})}<\frac{1}{4}.
\]
Letting \(E_{k}=X_{k}-X\) denote the iteration error,
this condition can be expressed as
\begin{equation}
\label{eq:suff-condition}
\norm{B_{X_{0}}^{-1}}\cdot
\bignorm{\VEC(E_{0})-B_{X_{0}}^{-1}\cdot\VEC(E_{0}^{2})}<\frac{1}{4}.
\end{equation}
Let \(\mu=2\norm{B_{X_{0}}^{-1}}\cdot\normfro{E_{0}}\). Note that
\[
4\norm{B_{X_{0}}^{-1}}\cdot\bignorm{\VEC(E_{0})-B_{X_{0}}^{-1}
\cdot\VEC(E_{0}^{2})}\leq4\norm{B_{X_{0}}^{-1}}\normfro{E_{0}}
+4\norm{B_{X_{0}}^{-1}}^{2}\normfro{E_{0}}^{2}=2\mu+\mu^{2}.
\]
Therefore if \(\mu<\sqrt{2}-1\), condition~\eqref{eq:suff-condition} holds,
and hence iteration~\eqref{eq:vec} converges.
In fact, this convergence condition can be further weakened.
This is implied by the following theorem,
which covers a more general form of \(B\) needed later.

\begin{theorem}
\label{thm:convergence}
With the notation introduced above, let \(Z\in\R^{n\times n}\)
be such that \(B_{Z}=I_{n}\otimes Z+Z\trans\otimes I_{n}\)
is invertible and define \(E_{Z}=Z-X\).
Then the error \(E_{k}=X_{k}-X\) of the sequence \(\{X_{k}\}\)
generated by~\eqref{eq:vec} with such a choice of \(B_{Z}\) satisfies
\[
\normfro{E_{k+1}}\leq 2\norm{B_{Z}^{-1}}\normfro{E_{Z}}
\normfro{E_{k}}+\norm{B_{Z}^{-1}}\normfro{E_{k}}^{2}. 
\]
Assuming 
\[
\rho_{0}=\norm{B_{Z}^{-1}}\bigl(2\normfro{E_{Z}}+\normfro{E_{0}}\bigr)<1,
\]
it holds that \(\normfro{E_{k+1}}\leq\rho_{0}\normfro{E_{k}}\),
that is, the error converges at least linearly with rate \(\rho_{0}\).
\end{theorem}

\begin{proof}
Note that
\[
A-X_{k}^{2}=X^{2}-X_{k}^{2}
=\frac{1}{2}\bigl[(X+X_{k})(X-X_{k})+(X-X_{k})(X+X_{k})\bigr].
\]
Define \(e_{k}=-\VEC(E_{k})\) as the vectorized iteration error. Then
\begin{align*}
e_{k+1} & =\VEC(X-X_{k+1})=\VEC(X-X_{k})-B_{Z}^{-1}\cdot\VEC(A-X_{k}^{2})\\
& =e_{k}-\frac{1}{2}\,B_{Z}^{-1}\cdot\VEC
\bigl[(X+X_{k})(X-X_{k})+(X-X_{k})(X+X_{k})\bigr]\\
& =e_{k}-\frac{1}{2}\,B_{Z}^{-1}\cdot\bigl[(I_{n}\otimes(X+X_{k}))
\VEC(X-X_{k})+((X+X_{k})\trans\otimes I_{n})\VEC(X-X_{k})\bigr]\\
& =\frac{1}{2}\,B_{Z}^{-1}\cdot\bigl[2B_{Z}-I_{n}\otimes(X+X_{k})
-(X+X_{k})\trans\otimes I_{n}\bigr]e_{k}\\
& =\frac{1}{2}\,B_{Z}^{-1}\cdot\bigl[I_{n}\otimes(2Z-X-X_{k})
+(2Z-X-X_{k})\trans\otimes I_{n}\bigr]e_{k}.
\end{align*}
Noting that \(2Z-X-X_{k}=2E_{Z}-E_{k}\), it follows that
\begin{align*}
\normfro{E_{k+1}} & =\norm{e_{k+1}}
\leq\frac{1}{2}\,\norm{B_{Z}^{-1}}\cdot\Bignorm{\bigl[I_{n}\otimes
(2E_{Z}-E_{k})+(2E_{Z}-E_{k})\trans\otimes I_{n}\bigr]e_{k}}\\
& =\frac{1}{2}\,\norm{B_{Z}^{-1}}\cdot\Bignormfro{
\bigl[(2E_{Z}-E_{k})E_{k}+E_{k}(2E_{Z}-E_{k})\bigr]}\\
& \leq\norm{B_{Z}^{-1}}\cdot\normfro{E_{k}}\cdot
\bigl(\normfro{E_{k}}+2\normfro{E_{Z}}\bigr).
\end{align*}
This shows the first part of the statement.
If \(\rho_{0}<1\), it can be shown by induction that
\(\normfro{E_{k}}\leq\normfro{E_{0}}\).
This implies the second part of the statement:
\(\normfro{E_{k+1}}\leq\rho_{0}\normfro{E_{k}}\).
\end{proof}

\subsection{Accuracy analysis for symmetric positive definite matrices}
\label{subsec:accuracy-spd}
In this section, we analyse the effect of rounding errors in
Algorithm~\ref{alg:spd} for a symmetric positive definite matrix~\(A\).
Specifically, we provide a convergence condition and a bound on the limiting
accuracy of Algorithm~\ref{alg:spd} in Theorem~\ref{thm:main-spd}.

The first step of Algorithm~\ref{alg:spd} is to
compute the spectral decomposition \(A=QTQ\trans\).
To simplify the analysis, we assume that
the standard symmetric QR algorithm is used for this purpose
with a constant number of QR iterations per eigenvalue.
Then the results from~\cite[Section 8.3]{GV2013}
and~\cite[Chapter 19]{Higham2002} suggest that the factor \(\hat{Q}\)
and the diagonal factor \(\hat{T}\) computed in lower precision satisfy
\begin{equation}
\label{eq:assump-Q}
A+E=\hat{Q}\hat{T}\hat{Q}\trans,
\end{equation}
where \(\norm{\hat{Q}\trans\hat{Q}-I_{n}}\leq\eiggamma(n,\machepsl)\),
and \(E\) is symmetric with \(\norm{E}\leq\eiggamma(n,\machepsl)\,\norm{A}\).
Here,
\begin{equation}
\label{eq:tgamma}
\eiggamma(n,\machepsl)=\frac{cn^{2}\machepsl}{1-cn^{2}\machepsl},
\end{equation}
where \(c\) is a positive constant independent
of \(n\), \(\machepsl\), and \(A\).
In the following, we tacitly assume that this error bound
is well defined, i.e., \(cn^{2}\machepsl<1\) and holds true.
We write \(\eiggamma\equiv\eiggamma(n,\machepsl)\) for simplicity.

The next theorem provides an upper bound for the rounding
error effected by the lower precision Schur algorithm,
i.e., lines 1--3 of Algorithm~\ref{alg:spd}.

\begin{theorem}
\label{thm:schur-spd}
Recall \(\gamma_{n}^{l}\) from~\eqref{eq:gamma}, \(\eiggamma\)
from~\eqref{eq:tgamma}, and assume that \(\eiggamma\leq\sqrt{5}-2\).
For a symmetric positive definite matrix \(A\), if \(\eiggamma\cdot\kappa_{2}(A)
\leq1/4\), the Schur algorithm in lines 1--3 of Algorithm~\ref{alg:spd}
carried out in lower precision returns a matrix \(\hat{X}_{0}\) that satisfies
\[
\normfro{\hat{X}_{0}-A^{1/2}}\leq2n^{1/2}(1+2n^{1/2}\,\gamma_{2n+1}^{l})
\eiggamma\cdot\norm{A^{-1}}^{1/2}\normfro{A}
+2n^{1/2}\,\gamma_{2n+1}^{l}\cdot\normfro{A^{1/2}}.
\]
\end{theorem}

\begin{proof}
From \(\norm{\hat{Q}\trans\hat{Q}-I_{n}}\leq\eiggamma\),
we see that \(\norm{\hat{Q}}^{2}\leq1+\eiggamma\) and
\(\norm{\hat{Q}^{-1}}^{2}\leq1/(1-\eiggamma)\). Thus
\begin{equation}
\label{eq:norm-T}
\norm{\hat{T}}=\norm{\hat{Q}^{-1}(A+E)\hat{Q}\itrans}\leq\norm{\hat{Q}^{-1}}^{2}
\cdot\norm{A+E}\leq\frac{1}{1-\eiggamma}\cdot\norm{A+E}.
\end{equation}

Note that \(\hat{T}\) is diagonal and positive definite, and,
hence, the square root \(\hat{S}=\hat{T}^{1/2}\) is diagonal
and positive definite as well, with the diagonal entries computed
through \(\hat{s}_{i,i}=\fl_{l}\bigl(\hat{t}_{i,i}^{1/2}\bigr)\).
By the Wilkinson model~\cite[Section 2.2]{Higham2002},
there exists a \(\delta_{i}\) such that
\[
\hat{s}_{i,i}=\hat{t}_{i,i}^{1/2}(1+\delta_{i}),\qquad
\abs{\delta_{i}}\leq\machepsl,\quad 1\leq i\leq n.
\]
Hence there exists a diagonal matrix \(H=\diag(\delta_{i})\) such that
\[
\hat{S}=\hat{T}^{1/2}(I_{n}+H),\qquad\norm{H}\leq\machepsl.
\]

In transforming \(\hat{X}_{0}=\fl_{l}(\hat{Q}\hat{S}\hat{Q}\trans)\),
there exists a matrix \(F\) such that
\[
\hat{X}_{0}=\hat{Q}\hat{S}\hat{Q}\trans+F,
\]
where the componentwise error bound \(\abs{F}\leq
\gamma_{2n}^{l}\abs{\hat{Q}}\abs{\hat{S}}\abs{\hat{Q}\trans}\)
in~\cite[Section 3.5]{Higham2002} can be used to imply that
\[
\normfro{F}\leq\gamma_{2n}^{l}(1+\machepsl)(1+\eiggamma)
\cdot n^{1/2}\cdot\normfro{\hat{T}^{1/2}}.
\]

Note that
\begin{equation}
\label{eq:X0}
\hat{X}_{0}-A^{1/2}=\hat{Q}\hat{S}\hat{Q}\trans+F-A^{1/2}
=\hat{Q}\hat{T}^{1/2}\hat{Q}\trans+\hat{Q}\hat{T}^{1/2}H\hat{Q}\trans+F-A^{1/2}.
\end{equation}
Let \(\hat{Q}\trans\hat{Q}=I_{n}+J\), with \(\norm{J}\leq\eiggamma\).
We can now choose \(G\) such that
\begin{equation}
\label{eq:error-G}
(A+E+G)^{1/2}=\hat{Q}\hat{T}^{1/2}\hat{Q}\trans,
\end{equation}
where by squaring both sides of~\eqref{eq:error-G},
\[
G=\hat{Q}\hat{T}^{1/2}(I_{n}+J)\hat{T}^{1/2}\hat{Q}\trans-A-E
=\hat{Q}\hat{T}^{1/2}J\hat{T}^{1/2}\hat{Q}\trans.
\]
Since \(0\leq\eiggamma\leq\sqrt{5}-2\), \((1+\eiggamma)^{2}\leq2(1-\eiggamma)\).
As \(\norm{E}\leq\eiggamma\,\norm{A}\), we have
\[
\norm{G}\leq\norm{\hat{Q}}^{2}\norm{J}\norm{\hat{T}}\leq
\frac{\eiggamma(1+\eiggamma)}{1-\eiggamma}\norm{A+E}\leq
\frac{\eiggamma(1+\eiggamma)^{2}}{1-\eiggamma}\norm{A}\leq2\eiggamma\,\norm{A}.
\]
Then \(\norm{E+G}\leq3\eiggamma\,\norm{A}\), and \(\normfro{E+G}\leq n^{1/2}
\norm{E+G}\leq3n^{1/2}\eiggamma\,\norm{A}\leq3n^{1/2}\eiggamma\,\normfro{A}\).
It follows from \(\norm{E+G}\leq3\eiggamma\,\norm{A}\leq3/4\cdot\norm{A^{-1}}
^{-1}=3/4\cdot\lambda_{\min}(A)\) that \(A+E+G\) is symmetric positive definite.
Then by~\cite[Theorem 6.2]{Higham2008},
\begin{align*}
\normfro{(A+E+G)^{1/2}-A^{1/2}} & \leq\frac{\normfro{E+G}}
{\lambda_{\min}^{1/2}(A)+\lambda_{\min}^{1/2}(A+E+G)}\\
& \leq\frac{\normfro{E+G}}{\lambda_{\min}^{1/2}(A)+(\lambda_{\min}(A)
-\norm{E+G})^{1/2}}\leq2n^{1/2}\eiggamma\cdot\norm{A^{-1}}^{1/2}\normfro{A}.
\end{align*}

Similarly to~\eqref{eq:norm-T}, since \(\hat{T}^{1/2}=\hat{Q}^{-1}
(A+E+G)^{1/2}\hat{Q}\itrans\), we see that \(\normfro{\hat{T}^{1/2}}
\leq\normfro{(A+E+G)^{1/2}}/(1-\eiggamma)\). This leads to
\[
\normfro{\hat{Q}\hat{T}^{1/2}H\hat{Q}\trans}\leq\frac{\machepsl(1+\eiggamma)}
{1-\eiggamma}\normfro{(A+E+G)^{1/2}}
\]
and
\[
\normfro{F}\leq\frac{\gamma_{2n}^{l}(1+\machepsl)(1+\eiggamma)}
{1-\eiggamma}\cdot n^{1/2}\cdot\normfro{(A+E+G)^{1/2}}.
\]
Therefore from~\eqref{eq:X0}, as \(1+\eiggamma\leq2(1-\eiggamma)\),
\begin{align*}
\normfro{\hat{X}_{0}-A^{1/2}} & \leq\normfro{(A+E+G)^{1/2}-A^{1/2}}
+\normfro{\hat{Q}\hat{T}^{1/2}H\hat{Q}\trans}+\normfro{F}\\
& \leq\normfro{(A+E+G)^{1/2}-A^{1/2}}+n^{1/2}\,\gamma_{2n+1}^{l}\cdot
\frac{1+\eiggamma}{1-\eiggamma}\cdot\normfro{(A+E+G)^{1/2}}\\
& \leq(1+2n^{1/2}\,\gamma_{2n+1}^{l})\cdot\normfro{(A+E+G)^{1/2}
-A^{1/2}}+2n^{1/2}\,\gamma_{2n+1}^{l}\cdot\normfro{A^{1/2}}\\
& \leq2n^{1/2}(1+2n^{1/2}\,\gamma_{2n+1}^{l})\eiggamma\cdot\norm{A^{-1}}
^{1/2}\normfro{A}+2n^{1/2}\,\gamma_{2n+1}^{l}\cdot\normfro{A^{1/2}}.
\end{align*}
\end{proof}

We then consider the iterative refinement process
in lines 4--13 of Algorithm~\ref{alg:spd}.
Define \(\tilde{X}_{0}=\hat{Q}\hat{S}\hat{Q}\trans\),
where the factor \(\hat{Q}\) and diagonal factor \(\hat{S}\) from
Theorem~\ref{thm:schur-spd} are the ones returned by the Schur algorithm
in lower precision and are later used in iterative refinement.

In Theorem~\ref{thm:error-spd}, we shall derive the relation
between the errors \(\{E_{k}\}\) in each iterative refinement step,
combining both the iteration errors bounded by Theorem~\ref{thm:convergence}
and additional rounding errors during iterative refinement.

\begin{theorem}
\label{thm:error-spd}
Recall \(\gamma_{n}, \gamma_{n}^{l}\) from~\eqref{eq:gamma} and \(\eiggamma\)
from~\eqref{eq:tgamma}, and assume that \(\eiggamma\leq\sqrt{5}-2\).
Consider a symmetric positive definite matrix \(A\) with \(\eiggamma\cdot
\kappa_{2}(A)<1\), and recall that \(\hat{X}_{0}\) and \(\{\hat{X}_{k}\}\)
are the matrices produced by the Schur algorithm and iterative refinement in
Algorithm~\ref{alg:spd}, respectively, carried out in lower/working precision.
Recall \(X=A^{1/2}\) and \(\tilde{X}_{0}\) defined before,
and let \(B_{\tilde{X}_{0}}=I_{n}\otimes\tilde{X}_{0}
+\tilde{X}_{0}\trans\otimes I_{n}\).
Then the errors \(E_{k}=\hat{X}_{k}-X\) satisfy the recursion
\[
\normfro{E_{k+1}}\leq\nu\normfro{E_{k}}+\beta\normfro{E_{k}}^{2}+\xi_{k},
\]
where
\[
\nu=2\bignorm{B_{\tilde{X}_{0}}^{-1}}
\bigl(\normfro{\tilde{X}_{0}-X}+2\epsilon_{l}\normfro{X}\bigr),
\qquad\beta=(1+2\epsilon_{l})\,\bignorm{B_{\tilde{X}_{0}}^{-1}},
\]
with \(\epsilon_{l}=\gamma_{n^{1/2}(4n+3)}^{l}
+4\eiggamma\cdot\kappa_{2}(B_{\tilde{X}_{0}})\), and
\[
\xi_{k}=2\gamma_{n+1}(1+\epsilon_{l})\,\bignorm{B_{\tilde{X}_{0}}^{-1}}\cdot
(\normfro{A}+\normfro{\hat{X}_{k}}^{2})+\gamma_{1}\normfro{\hat{X}_{k+1}}.
\]
\end{theorem}

\begin{proof}
The first step in each iterative refinement loop
is to compute the residual in working precision.
Let \(R_{k}=A-\hat{X}_{k}^{2}\).
Suppose \(\hat{P}_{k}=\fl(R_{k})=\fl(A-\hat{X}_{k}^{2})\).
Then there exists a \(\Delta P_{k}\) such that
\[
\hat{P}_{k}=R_{k}+\Delta P_{k},\qquad
\normfro{\Delta P_{k}}\leq\gamma_{n+1}(\normfro{A}+\normfro{\hat{X}_{k}}^{2}).
\]
The error bound can be derived from the matrix multiplication
error bounds in~\cite[Section~3.5]{Higham2002}.
This is then followed by converting \(\hat{P}_{k}\) to lower precision:
\(\hat{R}_{k}=\fl_{l}(\hat{P}_{k})=\hat{P}_{k}+\Delta F_{k}\),
with \(\normfro{\Delta F_{k}}\leq\machepsl\,\normfro{\hat{P}_{k}}\).
Setting \(\Delta R_{k}=\Delta P_{k}+\Delta F_{k}\), we obtain
\[
\hat{R}_{k}=R_{k}+\Delta P_{k}+\Delta F_{k}=R_{k}+\Delta R_{k},
\]
where
\[
\normfro{\Delta R_{k}}\leq\normfro{\Delta P_{k}}
+\normfro{\Delta F_{k}}\leq\machepsl\normfro{R_{k}}
+(1+\machepsl)\gamma_{n+1}(\normfro{A}+\normfro{\hat{X}_{k}}^{2}).
\]
Therefore
\[
\normfro{\hat{R}_{k}}\leq\normfro{R_{k}}
+\normfro{\Delta R_{k}}\leq(1+\machepsl)\,\normfro{R_{k}}
+(1+\machepsl)\gamma_{n+1}(\normfro{A}+\normfro{\hat{X}_{k}}^{2}).
\]

The next step consists of solving a (diagonal) Sylvester equation.
When forming the right-hand side \(\hat{Q}\trans\hat{R}_{k}\hat{Q}\),
similarly to the proof of Theorem~\ref{thm:schur-spd}, the componentwise
error bound in~\cite[Section 3.5]{Higham2002} can be used to
imply that there exists a \(\delta R_{k}\) that satisfies
\[
\fl_{l}(\hat{Q}\trans\hat{R}_{k}\hat{Q})
=\hat{Q}\trans\hat{R}_{k}\hat{Q}+\delta R_{k},\qquad\normfro{\delta R_{k}}
\leq n^{1/2}\,\gamma_{2n}^{l}\cdot\norm{\hat{Q}}^{2}\normfro{\hat{R}_{k}}.
\]

Note that the assumption \(\eiggamma\cdot\kappa_{2}(A)<1\) implies that
the computed diagonal square root \(\hat{S}\) is positive definite and, hence, 
\(B_{\hat{S}}=I_{n}\otimes\hat{S}+\hat{S}\otimes I_{n}\) is invertible.
Suppose that \(\bar{Y}_{k}\) is the exact solution to \(\hat{S}Y_{k}
+Y_{k}\hat{S}=\hat{Q}\trans\hat{R}_{k}\hat{Q}\), and that \(\hat{Y}_{k}\)
is the computed solution (in lower precision) to \(\hat{S}Y_{k}+Y_{k}\hat{S}
=\fl_{l}(\hat{Q}\trans\hat{R}_{k}\hat{Q})\), respectively.
By~\cite[Section 16.3]{Higham2002} and the Wilkinson model,
as there are two floating-point operations for computing each element,
\begin{align*}
\normfro{\hat{Y_{k}}-\bar{Y}_{k}}
& \leq\norm{B_{\hat{S}}^{-1}}\cdot\bigl(\gamma^l_{2}\,\normfro{\hat{Q}
\trans\hat{R}_{k}\hat{Q}+\delta R_{k}}+\normfro{\delta R_{k}}\bigr)\\
& \leq n^{1/2}\,(\gamma_{2}^{l}(1+\gamma_{2n}^{l})+\gamma_{2n}^{l})
\cdot\norm{B_{\hat{S}}^{-1}}\norm{\hat{Q}}^{2}\normfro{\hat{R}_{k}}\\
& \leq n^{1/2}\,\gamma_{2n+2}^{l}\cdot\norm{B_{\hat{S}}^{-1}}
\norm{\hat{Q}}^{2}\normfro{\hat{R}_{k}}.
\end{align*}

Similarly in forming \(\hat{Q}\hat{Y}\hat{Q}\trans\),
by~\cite[Section 3.5]{Higham2002}, there exists a \(\Delta Y_{k}\) such that
\[
\Delta\hat{X}_{k}=\fl_{l}(\hat{Q}\hat{Y}_{k}\hat{Q}\trans)
=\hat{Q}\hat{Y}_{k}\hat{Q}\trans+\Delta Y_{k},\qquad\normfro{\Delta Y_{k}}
\leq n^{1/2}\,\gamma_{2n}^{l}\cdot\norm{\hat{Q}}^{2}\normfro{\hat{Y}_{k}}.
\]
Denote \(\Delta\bar{X}_{k}=\hat{Q}\bar{Y}_{k}\hat{Q}\trans\).
Noting \(\normfro{\bar{Y}_{k}}\leq\norm{B_{\hat{S}}^{-1}}
\norm{\hat{Q}}^{2}\normfro{\hat{R}_{k}}\), we have
\begin{align}
\normfro{\Delta\hat{X}_{k}-\Delta\bar{X}_{k}} & \leq\normfro{\hat{Q}\hat{Y}_{k}
\hat{Q}\trans-\hat{Q}\bar{Y}_{k}\hat{Q}\trans}+\normfro{\Delta Y_{k}}\notag\\
& \leq(1+\eiggamma)\,\normfro{\hat{Y_{k}}-\bar{Y}_{k}}
+n^{1/2}\,\gamma_{2n}^{l}(1+\eiggamma)\cdot
\bigl(\normfro{\hat{Y_{k}}-\bar{Y}_{k}}+\normfro{\bar{Y}_{k}}\bigr)\notag\\
& \leq(1+n^{1/2}\,\gamma_{2n}^{l})(1+\eiggamma)\cdot
\normfro{\hat{Y_{k}}-\bar{Y}_{k}}+n^{1/2}\,\gamma_{2n}^{l}
(1+\eiggamma)\cdot\normfro{\bar{Y}_{k}}\notag\\
& \leq\gamma_{n^{1/2}(4n+2)}^{l}\,(1+\eiggamma)^{2}\cdot
\norm{B_{\hat{S}}^{-1}}\normfro{\hat{R}_{k}}\label{eq:bound-spd-1}.
\end{align}

Suppose \(\Delta\check{X}_{k}\) is the exact solution to
\(\tilde{X}_{0}\Delta X_{k}+\Delta X_{k}\tilde{X}_{0}=\hat{R}_{k}\).
Define \(\mathcal{L}(X)=\tilde{X}_0X+X\tilde{X}_0\).
Then \(\Delta\check{X}_{k}=\mathcal{L}^{-1}(\hat{R}_{k})\).
Let \(\hat{Q}\trans\hat{Q}=I+J\), \(\hat{Q}\hat{Q}\trans=I+K\),
with \(\norm{J}=\norm{K}\leq\eiggamma\). Then
\begin{align*}
\mathcal{L}(\Delta\bar{X}_{k}) & =
\hat{Q}\hat{S}\hat{Q}\trans\hat{Q}\bar{Y}_{k}\hat{Q}\trans
+\hat{Q}\bar{Y}_{k}\hat{Q}\trans\hat{Q}\hat{S}\hat{Q}\trans\\
& =\hat{Q}(\hat{S}\bar{Y}_{k}+\bar{Y}_{k}\hat{S})\hat{Q}\trans
+\hat{Q}(\hat{S}J\bar{Y}_{k}+\bar{Y}_{k}J\hat{S})\hat{Q}\trans\\
& =\hat{Q}\hat{Q}\trans\hat{R}_{k}\hat{Q}\hat{Q}\trans
+\hat{Q}(\hat{S}J\bar{Y}_{k}+\bar{Y}_{k}J\hat{S})\hat{Q}\trans\\
& =\hat{R}_{k}+K\hat{R}_{k}+\hat{R}_{k}K+K\hat{R}_{k}K
+\hat{Q}(\hat{S}J\bar{Y}_{k}+\bar{Y}_{k}J\hat{S})\hat{Q}\trans.
\end{align*}
Recall that \(\normfro{\bar{Y}_{k}}\leq\norm{B_{\hat{S}}^{-1}}
\norm{\hat{Q}}^{2}\normfro{\hat{R}_{k}}\). Hence
\begin{align*}
\normfro{\mathcal{L}(\Delta\bar{X}_{k})-\hat{R}_{k}}
& \leq(2\norm{K}+\norm{K}^{2})\,\normfro{\hat{R}_{k}}
+2\norm{\hat{Q}}^{2}\norm{\hat{S}}\norm{J}\normfro{\bar{Y}_{k}}\\
& \leq\bigl(2\eiggamma+\eiggamma^{2}+2\eiggamma(1+\eiggamma)^{2}\,
\norm{B_{\hat{S}}^{-1}}\norm{\hat{S}}\bigr)\cdot\normfro{\hat{R}_{k}}.
\end{align*}

Note that \(\tilde{X}_{0}=\hat{Q}\hat{S}\hat{Q}\trans\) is
similar to \(\hat{S}^{1/2}\hat{Q}\trans\hat{Q}\hat{S}^{1/2}\).
Hence
\[
\lambda_{\min}(\tilde{X}_{0})=\lambda_{\min}(\hat{S}^{1/2}\hat{Q}
\trans\hat{Q}\hat{S}^{1/2})\leq(1+\eiggamma)\,\lambda_{\min}(\hat{S}).
\]
As \(B_{\tilde{X}_{0}}=I_{n}\otimes\tilde{X}_{0}
+\tilde{X}_{0}\trans\otimes I_{n}\), it follows that
\[
\norm{B_{\hat{S}}^{-1}}=\frac{1}{\lambda_{\min}(B_{\hat{S}})}
=\frac{1}{2\lambda_{\min}(\hat{S})}\leq\frac{1+\eiggamma}{2\lambda_{\min}
(\tilde{X}_{0})}=\frac{1+\eiggamma}{\lambda_{\min}(B_{\tilde{X}_{0}})}
=(1+\eiggamma)\norm{B_{\tilde{X}_{0}}^{-1}}.
\]
Hence similarly to~\eqref{eq:norm-T},
\[
1\leq\kappa_{2}(\hat{S})=2\norm{B_{\hat{S}}^{-1}}\norm{\hat{S}}
\leq2(1+\eiggamma)
\norm{B_{\tilde{X}_{0}}^{-1}}\cdot\frac{\norm{\tilde{X}_{0}}}{1-\eiggamma}
\leq\frac{1+\eiggamma}{1-\eiggamma}\cdot\kappa_{2}(B_{\tilde{X}_{0}}).
\]

As \(\Delta\check{X}_{k}-\Delta\bar{X}_{k}=\mathcal{L}^{-1}
\bigl(\hat{R}_{k}-\mathcal{L}(\Delta\bar{X}_{k})\bigr)\),
and \(3\eiggamma+\eiggamma^{2}<1\), by~\cite[Section~16.3]{Higham2002},
\begin{align}
\normfro{\Delta\check{X}_{k}-\Delta\bar{X}_{k}} & \leq\norm{B_{\tilde{X}_{0}}
^{-1}}\cdot\normfro{\mathcal{L}(\Delta\bar{X}_{k})-\hat{R}_{k}}\notag\\
& \leq\bigl(2\eiggamma+\eiggamma^{2}+\eiggamma(1+\eiggamma)^{2}\,\kappa_{2}
(\hat{S})\bigr)\cdot\norm{B_{\tilde{X}_{0}}^{-1}}\normfro{\hat{R}_{k}}\notag\\
& \leq(3\eiggamma+3\eiggamma^{2}+\eiggamma^{3})\cdot\kappa_{2}(\hat{S})
\cdot\norm{B_{\tilde{X}_{0}}^{-1}}\normfro{\hat{R}_{k}}\notag\\
& \leq\frac{4\eiggamma(1+\eiggamma)}{1-\eiggamma}\cdot\kappa_{2}
(B_{\tilde{X}_{0}})\cdot\norm{B_{\tilde{X}_{0}}^{-1}}\normfro{\hat{R}_{k}}
\label{eq:bound-spd-2}
\end{align}

Suppose \(\Delta\tilde{X}_{k}\) is the exact solution to
\(\tilde{X}_{0}\Delta X_{k}+\Delta X_{k}\tilde{X}_{0}=R_{k}\). Then
\begin{equation}
\label{eq:bound-spd-3}
\normfro{\Delta\tilde{X}_{k}-\Delta\check{X}_{k}}\leq
\norm{B_{\tilde{X}_{0}}^{-1}}\cdot\normfro{R_{k}-\hat{R}_{k}}
=\norm{B_{\tilde{X}_{0}}^{-1}}\normfro{\Delta R_{k}}.
\end{equation}

The final updating step in working precision can be expressed as
\[
\hat{X}_{k+1}=\fl(\hat{X}_{k}+\Delta\hat{X}_{k})
=\hat{X}_{k}+\Delta\hat{X}_{k}+\delta\hat{X}_{k},\qquad
\normfro{\delta\hat{X}_{k}}\leq\gamma_{1}\normfro{\hat{X}_{k+1}}.
\]
This is equivalent to
\[
\hat{X}_{k+1}=\hat{X}_{k}+\Delta\tilde{X}_{k}+(\Delta\check{X}_{k}
-\Delta\tilde{X}_{k})+(\Delta\bar{X}_{k}-\Delta\check{X}_{k})
+(\Delta\hat{X}_{k}-\Delta\bar{X}_{k})+\delta\hat{X}_{k}.
\]

As \(\eiggamma\leq(1+\eiggamma)(1-\eiggamma)\), the rounding error
\(G_{k}=\hat{X}_{k+1}-\hat{X}_{k}-\Delta\tilde{X}_{k}\),
by~\eqref{eq:bound-spd-1},~\eqref{eq:bound-spd-2}, and~\eqref{eq:bound-spd-3},
satisfies
\begin{align*}
\normfro{G_{k}} & \leq\normfro{\Delta\tilde{X}_{k}-\Delta\check{X}_{k}}
+\normfro{\Delta\bar{X}_{k}-\Delta\check{X}_{k}}
+\normfro{\Delta\hat{X}_{k}-\Delta\bar{X}_{k}}+\normfro{\delta\hat{X}_{k}}\\
& \leq\norm{B_{\tilde{X}_{0}}^{-1}}\,\bigl(\normfro{\Delta R_{k}}
+\bigl(\gamma_{n^{1/2}(4n+2)}^{l}(1+\eiggamma)^{3}+\frac{4\eiggamma
(1+\eiggamma)}{1-\eiggamma}\kappa_{2}(B_{\tilde{X}_{0}})\bigr)
\normfro{\hat{R}_{k}}\bigr)+\gamma_{1}\normfro{\hat{X}_{k+1}}\\
& \leq\frac{(1+\eiggamma)^{2}}{1-\eiggamma}\cdot\norm{B_{\tilde{X}_{0}}^{-1}}
\cdot\bigl(\epsilon_{l}\,\normfro{R_{k}}+\gamma_{n+1}(1+\epsilon_{l})
(\normfro{A}+\normfro{\hat{X}_{k}}^{2})\bigr)+\gamma_{1}\normfro{\hat{X}_{k+1}}.
\end{align*}

Note that
\[
\normfro{R_{k}}=\normfro{\hat{X}_{k}^{2}-X^{2}}=\normfro{XE_{k}
+E_{k}\hat{X}_{k}}\leq(\normfro{X}+\normfro{\hat{X}_{k}})\normfro{E_{k}}
\leq2\normfro{X}\normfro{E_{k}}+\normfro{E_{k}}^{2}.
\]
As \(0\leq\eiggamma\leq\sqrt{5}-2\), \((1+\eiggamma)^{2}\leq2(1-\eiggamma)\).
Therefore,
\begin{equation}
\label{eq:G}
\normfro{G_{k}}\leq4\epsilon_{l}\,\norm{B_{\tilde{X}_{0}}^{-1}}
\normfro{X}\normfro{E_{k}}
+2\epsilon_{l}\,\norm{B_{\tilde{X}_{0}}^{-1}}\normfro{E_{k}}^{2}+\xi_{k}.
\end{equation}

Recall that \(\hat{X}_{k+1}=(\hat{X}_{k}+\Delta\tilde{X}_{k})+G_{k}\).
Applying Theorem~\ref{thm:convergence} with \(Z=\tilde{X}_{0}\) gives 
\[
\normfro{(\hat{X}_{k}+\Delta\tilde{X}_{k})-X}\leq
2\norm{B_{\tilde{X}_{0}}^{-1}}\cdot\normfro{\tilde{X}_{0}-X}\cdot
\normfro{E_{k}}+\norm{B_{\tilde{X}_{0}}^{-1}}\cdot\normfro{E_{k}}^{2}.
\]
Combining this with~\eqref{eq:G}, we obtain
\begin{align*}
\normfro{E_{k+1}} & \leq
2\norm{B_{\tilde{X}_{0}}^{-1}}\cdot\normfro{\tilde{X}_{0}-X}\cdot\normfro{E_{k}}
+\norm{B_{\tilde{X}_{0}}^{-1}}\cdot\normfro{E_{k}}^{2}+\normfro{G_{k}}\\
& \leq2\norm{B_{\tilde{X}_{0}}^{-1}}\cdot\bigl(\normfro{\tilde{X}_{0}-X}
+2\epsilon_{l}\normfro{X}\bigr)\cdot\normfro{E_{k}}+(1+2\epsilon_{l})\,
\norm{B_{\tilde{X}_{0}}^{-1}}\cdot\normfro{E_{k}}^{2}+\xi_{k}.
\end{align*}
This proves the claimed recursion.
\end{proof}

\begin{remark}
The framework of the proof is similar to those of previous works,
where we sequentially analyse the computational procedures and
then use the triangle inequality to derive a forward error relation.
The complexity of the proof is due to the fact that
we do not assume the orthogonality of \(\hat{Q}\)
but instead utilize a parameter \(\eiggamma\), which is more rigorous.
The analysis in Theorems~\ref{thm:schur-spd} and~\ref{thm:error-spd}
relies on \(\eiggamma<1\).
The stronger requirement \(\eiggamma\leq\sqrt{5}-2\)
is made primarily to simplify the analysis.
For a slightly larger \(\eiggamma\), there is still a similar result,
though certain constants will be larger.
\end{remark}

\begin{remark}
In Theorem~\ref{thm:error-spd}, we utilize \(\hat{X}_{0}\)
and \(\tilde{X}_{0}=\hat{Q}\hat{S}\hat{Q}\trans\)
in coherence with the procedures in Algorithm~\ref{alg:spd}.
Specifically, \(\hat{X}_{0}\) is employed as the first iterate of the computed
matrices \(\{\hat{X}_{k}\}\), and factors \(\hat{Q}\) and \(\hat{S}\) (which
collectively form \(\tilde{X}_{0}\)) are used to solve the correction equation.
Both \(\hat{X}_{0}\) and \(\tilde{X}_{0}\) can be seen as
different approximations of \(X_{0}\) in finite arithmetic.
\end{remark}

To turn the recursive inequality of Theorem~\ref{thm:error-spd}
into an explicit bound, we will make use of the following lemma.

\begin{lemma}
\label{lemma:sequence}
Consider the nonnegative scalar sequence \(\{x_{k}\}\): \(x_{k+1}=\alpha x_{k}
+\theta x_{k}^{2}+\zeta\), where \(\alpha\), \(\theta\), \(\zeta>0\),
\(\alpha<1\), and the discriminant \(\Delta=(1-\alpha)^{2}-4\theta\zeta>0\).
Let \(0<x_{L}<x_{R}\) be the solutions to
\[
x=\alpha x+\theta x^{2}+\zeta.
\]
If \(0\leq x_{0}<x_{R}\), the sequence \(\{x_{k}\}\)
converges to \(x_{L}\) at least linearly with
\[
\abs{x_{k}-x_{L}}\leq\rho^{k}\abs{x_{0}-x_{L}},
\]
where \(\rho=\max\bigl(\theta(x_{0}+x_{L})+\alpha, 1-\sqrt{\Delta}\,\bigr)<1\).
\end{lemma}

\begin{proof}
Let \(\delta_{k}=x_{k}-x_{L}\). Then
\[
\delta_{k+1}=x_{k+1}-x_{L}=\alpha(x_{k}-x_{L})+\theta(x_{k}^{2}-x_{L}^{2})
=(\alpha+\theta(x_{k}+x_{L}))\,\delta_{k}.
\]
By induction it is easy to show that \(x_{k}<x_{R}\) for all \(k\geq0\).
Let \(\lambda_{k}=\alpha+\theta(x_{k}+x_{L})>0\), where
\[
\lambda_{k}<\alpha+\theta(x_{R}+x_{L})
=\alpha+\theta\cdot\frac{1-\alpha}{\theta}=1.
\]
Define \(\rho:=\sup\nolimits_{j\geq 0}\lambda_{j}\). Then \(\rho\leq 1\).
Using \(\abs{\delta_{k+1}}=\lambda_{k}\abs{\delta_{k}}\),
we have \(\abs{\delta_{k}}\leq\rho^{k}\abs{\delta_{0}}\).

\begin{spacing}{1.3}
\begin{enumerate}
\item  If \(x_{L}\leq x_{0}<x_{R}\), \(\{x_{k}\}\) is decreasing.
In this case \(\rho=\alpha+\theta(x_{0}+x_{L})\).
\item  If \(0\leq x_{0}<x_{L}\), \(\{x_{k}\}\) is increasing.
In this case \(\rho=\alpha+2\theta x_{L}=1-\sqrt{(1-\alpha)^{2}-4\theta\zeta}\).
\end{enumerate}
\end{spacing}
Either case, \(\rho=\max\bigl(\theta(x_{0}+x_{L})+\alpha,
1-\sqrt{(1-\alpha)^{2}-4\theta\zeta}\,\bigr)<1\) as claimed.
\end{proof}

From Lemma~\ref{lemma:sequence}, we see that if
\(0\leq x_{0}\leq(1-\alpha)/2\theta<x_{R}\), for \(k\geq0\),
\begin{align}
x_{k} \leq x_{L}+\rho^{k}\abs{x_{0}-x_{L}}
& =\frac{1-\alpha-\sqrt{\Delta}}{2\theta}+\rho^{k}\abs{x_{0}-x_{L}}\notag\\
& =\frac{2\zeta}{1-\alpha+\sqrt{\Delta}}+\rho^{k}\abs{x_{0}-x_{L}}
\leq\frac{2\zeta}{1-\alpha}+\rho^{k}\abs{x_{0}-x_{L}}\label{eq:lemma-scalar}.
\end{align}
Now, assume that the computed iterates \(\{\hat{X}_{k}\}\) remain bounded, i.e.,
\(M=\sup_{k\geq0}\normfro{\hat{X}_{k}}/\normfro{X}<+\infty\).
Define
\[
\epsilon:=4\gamma_{n+1}(1+M^{2})\,\norm{B_{\tilde{X}_{0}}^{-1}}
\cdot\normfro{X}+\gamma_{1}M<+\infty.
\]
Then if \(\epsilon_{l}\leq1\), for \(k\geq0\),
\[
\xi_{k}=2\gamma_{n+1}(1+\epsilon_{l})\,\bignorm{B_{\tilde{X}_{0}}^{-1}}
\cdot(\normfro{X^{2}}+\normfro{\hat{X}_{k}}^{2})
+\gamma_{1}\normfro{\hat{X}_{k+1}}\leq\epsilon\cdot\normfro{X}.
\]

Let \(\alpha=\nu\), \(\theta=\beta\), and \(\zeta=\epsilon\cdot\normfro{X}\),
where \(\nu\), \(\beta\) are defined in Theorem~\ref{thm:error-spd}
and \(\epsilon\) is defined above.
Then from Theorem~\ref{thm:error-spd}, we see that sequence
\(\bigl\{\normfro{E_{k}}\bigr\}\) is no larger than the sequence
in Lemma~\ref{lemma:sequence} with \(\alpha\), \(\theta\), and \(\zeta\).
Hence by~\eqref{eq:lemma-scalar}, if \(\Delta=(1-\nu)^{2}-4\beta\epsilon\cdot
\normfro{X}>0\), \(\nu<1\), and \(\normfro{E_{0}}\leq(1-\nu)/2\beta\), then
\begin{equation}
\label{eq:limit}
\frac{\normfro{E_{k}}}{\normfro{X}}\leq\frac{2\epsilon}{1-\nu}
+\rho^{k}\cdot\frac{\abs{x_{0}-x_{L}}}{\normfro{X}},
\end{equation}
where \(x_{0}=\normfro{E_{0}}\),
\(x_{L}=\bigl(1-\nu-\sqrt{\Delta}\,\bigr)/2\beta\), and
\(\rho=\max\bigl(\beta(x_{0}+x_{L})+\nu,1-\sqrt{\Delta}\,\bigr)<1\).

Theorem~\ref{thm:error-spd} and Lemma~\ref{lemma:sequence}
together give a sufficient condition for iterative refinement
to achieve a forward error on the level of working precision.
We summarize a simplified version in the following theorem.

\begin{theorem}
\label{thm:main-spd}
Recall \(\gamma_{n}\), \(\gamma_{n}^{l}\)
from~\eqref{eq:gamma} and \(\eiggamma\) from~\eqref{eq:tgamma}.
For symmetric positive definite matrix \(A\),
recall that \(\hat{X}_{0}\) is the matrix computed by
the Schur algorithm in lines 1--3 of Algorithm~\ref{alg:spd},
\(X=A^{1/2}\) and \(\tilde{X}_{0}\) defined before Theorem~\ref{thm:error-spd}.
Let \(B_{\tilde{X}_{0}}=I_{n}\otimes\tilde{X}_{0}
+\tilde{X}_{0}\trans\otimes I_{n}\).
Assume that matrices \(\{\hat{X}_{k}\}\) returned
in each iterative refinement step are bounded, i.e.,
\(M=\sup_{k\geq0}\normfro{\hat{X}_{k}}/\normfro{X}<+\infty\), and define
\[
\epsilon:=4\gamma_{n+1}(1+M^{2})\,
\norm{B_{\tilde{X}_{0}}^{-1}}\cdot\normfro{X}+\gamma_{1}M.
\]
If \(A\) is well-conditioned such that
\[
n^{1/2}\,\eiggamma\cdot\norm{A^{-1}}\normfro{A}\leq\frac{1}{52},\quad
\epsilon_{l}\cdot\norm{A^{-1/2}}\normfro{A^{1/2}}\leq\frac{1}{12},
\quad\epsilon\cdot\norm{A^{-1/2}}\normfro{A^{1/2}}<\frac{3}{32},
\]
where \(\epsilon_{l}=\gamma_{n^{1/2}(4n+3)}^{l}+4\eiggamma\cdot\kappa_{2}
(B_{\tilde{X}_{0}})\), then \(\{\hat{X}_{k}\}\) converges, up to the limiting
accuracy \(4\epsilon\), to the (principal) square root \(X\) at least linearly.
Specifically, it holds that
\[
\frac{\normfro{X-\hat{X}_{k}}}{\normfro{X}}\leq4\epsilon+\alpha\rho^{k},
\]
where \(\alpha\geq0\), \(0\leq\rho<1\) are constants.
\end{theorem}

\begin{proof}
Recall \(E_{0}=\hat{X}_{0}-X\) and set \(\tilde{E}_{0}=\tilde{X}_{0}-X\).
As the conditions of Theorem~\ref{thm:schur-spd} are fulfilled,
and \(2n^{1/2}\,\gamma_{2n+1}^{l}\leq\epsilon_{l}\leq1/12\),
\begin{align}
\normfro{E_{0}} & \leq2n^{1/2}(1+2n^{1/2}\,\gamma_{2n+1}^{l})\eiggamma
\cdot\norm{A^{-1}}
^{1/2}\normfro{A}+2n^{1/2}\,\gamma_{2n+1}^{l}\cdot\normfro{A^{1/2}}\notag\\
& \leq\Bigl(\frac{1}{24}+\frac{1}{12}\Bigr)\,\norm{A^{-1}}^{-1/2}
\leq\frac{1}{8}\,\norm{A^{-1}}^{-1/2}\label{eq:E0}.
\end{align}
From~\eqref{eq:X0}, we see that \(\normfro{\tilde{X}_{0}-X}\) is also bounded
by the bound in Theorem~\ref{thm:schur-spd}, and thus
\[
\normfro{\tilde{X}_{0}-X}\leq\frac{1}{8}\norm{A^{-1}}^{-1/2}.
\]

Denote \(B_{X}=I_{n}\otimes X+X\trans\otimes I_{n}\).
Then \(\lambda_{\min}(B_{X})=2\lambda_{\min}(X)=2\cdot\lambda_{\min}^{1/2}(A)\).
Hence \(\norm{B_{X}^{-1}}=1/2\cdot\lambda_{\min}^{-1/2}(A)
=1/2\cdot\norm{A^{-1}}^{1/2}\). Note that
\[
\norm{B_{X}^{-1}}\norm{B_{\tilde{X}_{0}}-B_{X}}
=\norm{B_{X}^{-1}}\norm{I_{n}\otimes(\tilde{X}_{0}-X)+(\tilde{X}_{0}-X)\trans
\otimes I_{n}}\leq2\norm{B_{X}^{-1}}\normfro{\tilde{X}_{0}-X}\leq\frac{1}{8}.
\]
By~\cite[Lemma 2.3.3]{GV2013}, as \(B_{\tilde{X}_{0}}^{-1}
=(I_{n}+B_{X}^{-1}(B_{\tilde{X}_{0}}-B_{X}))^{-1}B_{X}^{-1}\), 
\[
\norm{B_{\tilde{X}_{0}}^{-1}}\leq\frac{\norm{B_{X}^{-1}}}
{1-\norm{B_{X}^{-1}}\norm{B_{\tilde{X}_{0}}-B_{X}}}
\leq\frac{4}{7}\,\norm{A^{-1}}^{1/2}.
\]

As the conditions of Theorem~\ref{thm:error-spd} are fulfilled,
and \(\epsilon_{l}\leq1/12\),
\begin{align*}
2\beta\,\normfro{E_{0}}+\nu & =2(1+2\epsilon_{l})\,
\norm{B_{\tilde{X}_{0}}^{-1}}\normfro{E_{0}}+2\norm{B_{\tilde{X}_{0}}^{-1}}
\bigl(\normfro{\tilde{X}_{0}-X}+2\epsilon_{l}\,\normfro{X}\bigr)\\
& \leq \frac12 (1+\epsilon_{l})\,\norm{B_{\tilde{X}_{0}}^{-1}}
\norm{A^{-1}}^{-1/2}+4\epsilon_{l}\,\norm{B_{\tilde{X}_{0}}^{-1}}
\normfro{X}\leq\frac{13}{42}+\frac{4}{21}=\frac{1}{2},
\end{align*}
and
\[
\beta\epsilon\cdot\normfro{X}=(1+2\epsilon_{l})\,\epsilon
\cdot\norm{B_{\tilde{X}_{0}}^{-1}}\normfro{X}\leq\frac{2}{3}\,\epsilon\cdot
\norm{A^{-1}}^{1/2}\normfro{A^{1/2}}<\frac{1}{16}.
\]
Then \((1-\nu)^{2}\geq1/4>4\beta\epsilon\cdot\normfro{X}\).
In other words, \(\nu\leq1/2<1\), \(\Delta>0\),
and \(2\beta\,\normfro{E_{0}}+\nu\leq1/2\).
Hence \(\normfro{E_{0}}\leq(1-\nu)/2\beta\).
From~\eqref{eq:limit} we see that \(\hat{X}_{k}\)
converges to \(X\) at least linearly with
\[
\frac{\normfro{X-\hat{X}_{k}}}{\normfro{X}}\leq\frac{2\epsilon}{1-\nu}+\rho^{k}
\cdot\frac{\abs{x_{0}-x_{L}}}{\normfro{X}}\leq4\epsilon+\alpha\rho^{k},
\]
where \(\alpha\geq0\), \(0\leq\rho<1\) are constants as in~\eqref{eq:limit}.
\end{proof}

\begin{remark}
The boundedness assumption on \(\{\hat{X}_{k}\}\) is mild.
Note that \(\normfro{\hat{X}_{k}}\leq\normfro{X}+\normfro{E_{k}}\).
Applying this relation to Theorem~\ref{thm:error-spd},
one obtains a scalar sequence for \(\normfro{E_{k}}\).
Consequently, if \(\normfro{E_{0}}\)
lies in the attraction interval of that sequence,
\(\{E_{k}\}\) is uniformly bounded, and so is~\(\{\hat{X}_{k}\}\).
\end{remark}

\begin{remark}
Higham derived in~\cite[Lemma 6.1]{Higham2008} the condition number
of matrix square roots for symmetric positive definite matrices,
and presented in~\cite[Equation (6.10)]{Higham2008} a backward error estimate
of the solution by the fixed precision Schur algorithm.
Combining these results with Theorem~\ref{thm:main-spd}, it is implied that
our upper bound \(\epsilon\) roughly matches the forward error bound of
matrix square roots computed by the fixed precision Schur algorithm.
\end{remark}

Note that with regard to machine precisions,
\(\epsilon=\mathcal{O}(\machepso)\).
Theorem~\ref{thm:main-spd} shows that for symmetric positive definite matrices,
if they are well-conditioned, Algorithm~\ref{alg:spd} converges to the matrix
square root at least linearly with an error no greater than \(4\epsilon\),
which is approximately at the same magnitude
as the unit roundoff of the working precision.
Furthermore, asymptotically,
\[
\nu=\mathcal{O}\bigl(\gamma_{n^{2}}^{l}\cdot
(\norm{A^{-1}}\normfro{A}+\norm{A^{-1/2}}\normfro{A^{1/2}})\bigr),
\qquad\beta=\mathcal{O}\bigl(\norm{A^{-1/2}}\bigr).
\]
Thus if \(A\) is well-conditioned and the lower machine precision
is not excessively low, the discriminant \((1-\nu)^{2}
-4\beta\epsilon\cdot\normfro{\hat{X}}\) is close to \(1\).
As \(x_{0}\), \(x_{L}\), and \(\nu\) are close to \(0\),
the rate of descent \(\rho\) is thus close to \(0\).
This implies a fast linear convergence rate.
The rate of descent \(\rho\) may grow larger
as \(A\) becomes more ill-conditioned.

\subsection{Accuracy analysis for general matrices}
\label{subsec:accuracy-general}
In this section, we analyse the effect of rounding errors
in Algorithm~\ref{alg:main} for a general matrix \(A\).
The spectral decomposition in Section~\ref{subsec:accuracy-spd} becomes
the Schur decomposition, and computing the square root of a diagonal
matrix turns into computing the square root of a triangular matrix.
Because of these changes, it is much more difficult to
derive the forward error of the computed approximation of
the Schur algorithm in lines 1--3 of Algorithm~\ref{alg:main},
as in~\cite[Algorithm 6.5]{Higham2008}; and no literature has presented
an upper bound for this forward error to the best of our knowledge.
Hence in this section, we mainly focus on the convergence
condition of the iterative refinement process in lines 4--13
of Algorithm~\ref{alg:main} \emph{in floating-point arithmetic}.

We first make an assumption concerning the solution of Sylvester equations.
For this purpose, let us recall that \(\hat{X}_{0}\) is the approximate
matrix square root returned by the Schur algorithm and
\[
\tilde{X}_{0}=\hat{Q}\hat{S}\hat{Q}\trans,\qquad
\hat{X}_{0}=\fl_{l}(\tilde{X}_{0}),\qquad\tilde{E}_{0}=\tilde{X}_{0}-X.
\]

\begin{assumption}
\label{assumption:forward}
Let \(B_{\tilde{X}_{0}}=I_{n}\otimes\tilde{X}_{0}+\tilde{X}_{0}\trans
\otimes I_{n}\) be invertible and \(\tilde{K}\) denote the exact
solution of the Sylvester equation \(\hat{X}_{0}\cdot\tilde{K}
+\tilde{K}\cdot\hat{X}_{0}=R\) for arbitrary \(R\in\R^{n\times n}\).
Let \(\hat{K}\) denote the solution computed in lower precision
by the Sylvester solver used in Algorithm~\ref{alg:main}.
Then it is assumed that there exists a constant \(h\equiv h(n,\machepsl,
\hat{X}_{0})\), independent of \(\hat{K}\), \(\tilde{K}\) and \(R\), such that
\[
\normfro{\hat{K}-\tilde{K}}\leq h\cdot\normfro{\tilde{K}}.
\]
\end{assumption}

Higham has provided some insight into \(h\); see specifically
equations (16.9) and (16.28) in~\cite[Chapter 16]{Higham2002}.
Under Assumption~\ref{assumption:forward},
the analysis in Theorem~\ref{thm:error-spd} carries over to the general
matrix case, with the lower precision error analysis for the Sylvester
solve replaced by the bound in Assumption~\ref{assumption:forward}.
The resulting error recursion is summarized in the following theorem.

\begin{theorem}
\label{thm:main-general}
Consider a general matrix \(A\), and recall that \(\hat{X}_{0}\)
and \(\{\hat{X}_{k}\}\) are the matrices produced by the Schur algorithm
and iterative refinement of Algorithm~\ref{alg:main}, respectively,
carried out in lower/working precision.
Let \(X=A^{1/2}\) and \(\tilde{X}_{0}\),
\(B_{\tilde{X}_{0}}\) be defined as above.
Assume that \(B_{\tilde{X}_{0}}\) is invertible and
that Assumption~\ref{assumption:forward} holds.
Then the errors \(E_{k}=\hat{X}_{k}-X\) satisfy the recursion
\[
\normfro{E_{k+1}}\leq\nu\normfro{E_{k}}+\beta\normfro{E_{k}}^{2}+\xi_{k},
\]
where
\[
\nu=2\norm{B_{\tilde{X}_{0}}^{-1}}\cdot\bigl(\normfro{\tilde{X}_{0}-X}
+h(1+\machepsl)\,\normfro{X}\bigr),\qquad
\beta=(1+h(1+\machepsl))\,\norm{B_{\tilde{X}_{0}}^{-1}},
\]
and
\[
\xi_{k}=h(1+\machepsl)\gamma_{n+1}\,\norm{B_{\tilde{X}_{0}}^{-1}}\cdot
(\normfro{A}+\normfro{\hat{X}_{k}}^{2})+\gamma_{1}\normfro{\hat{X}_{k+1}}.
\]
\end{theorem}

\begin{proof}
The proof is similar to and follows directly from the proof of
Theorem~\ref{thm:error-spd}; the only significant change is that
the estimates for the diagonal Sylvester solver,
namely~\eqref{eq:bound-spd-1}--\eqref{eq:bound-spd-2}, are replaced with
Assumption~\eqref{assumption:forward} applied to the correction equation.
\end{proof}

The conditions of Lemma 3.5 can be fulfilled when \(\tilde{X}_{0}\) is
sufficiently close to \(X\), the actual initial iterate \(\hat{X}_{0}\) lies in
the corresponding attraction region, and the computed iterates remain bounded.
Under these conditions, one may conclude that
the iterative refinement process converges --- up to
an error of order \(\machepso\) --- to the matrix square root.
This suggests that we have reason to believe that Algorithm~\ref{alg:main}
can attain working precision accuracy for sufficiently well-conditioned
problems and sufficiently stable Sylvester solvers. 
\section{Matrix \(p\)th roots}
\label{sec:pthroot}
In this section, we briefly extend our algorithmic framework
to matrix \(p\)th roots~\cite[Chapter 7]{Higham2008},
whose computation shares many similarities with that of matrix square roots.
Given a matrix \(A\in\C^{n\times n}\) with no eigenvalues on the non-positive
real axis \(\R_{\leq 0}\) and integer \(p\geq2\), the \emph{(principal)
matrix \(p\)th root} \(A^{1/p}\) of \(A\) is the unique matrix
\(X\in\C^{n\times n}\) such that \(X^{p}=A\) and all eigenvalues of \(X\)
lie in the cone region of \(\{z\in\C\backslash\{0\}:-\pi/p<\arg z<\pi/p\}\).

Similarly to Section~\ref{sec:algorithm}, for the computation of matrix
\(p\)th roots, given an approximation \(X_{0}\), the corrected matrix
\(X_{0}+\Delta X\) ideally satisfies \((X_{0}+\Delta X)^{p}=A\).
Neglecting the higher powers of \(\Delta X\), we derive the correction equation
\begin{equation}
\label{eq:pth-correction}
\sum_{i=1}^{p}X_{0}^{p-i}\cdot\Delta X\cdot X_{0}^{i-1}=R,
\end{equation}
which is a generalized matrix Sylvester equation.
Under the stated assumptions on \(A\),
the Fr\'{e}chet differentiability of matrix \(p\)th roots implies that
this equation has a unique solution for \(X_{0}\) sufficiently close to \(X\).

The mixed precision algorithm framework based on
the approximate Newton method is similarly presented as follows:

\begin{spacing}{1.3}
\begin{enumerate}
\item  Compute an initial approximation \(X_{0}\) of the matrix
\(p\)th root of \(A\) in lower precision and set \(k\gets 0\);
\item  Compute the residual \(R_{k}=A-X_{k}^{p}\) in working precision;
\item  Compute \(\Delta X_{k}\) by solving a correction equation
of the form~\eqref{eq:pth-correction} in lower precision;
\item  Add \(X_{k+1}\gets X_{k}+\Delta X_{k}\)
in working precision and update \(k\gets k+1\).

Repeat steps 2--4 until the residual \(\normfro{R_{k}}\) is sufficiently small.
\end{enumerate}
\end{spacing}

The Schur algorithm for matrix \(p\)th roots in~\cite[Section 7.2]{Higham2008}
is used in Step 1 of this framework, which follows a similar
routine as the Schur algorithm for matrix square roots.

In Step 3, a block recursion algorithm for solving generalized
triangular Sylvester equations~\eqref{eq:pth-correction}
as in Algorithm~\ref{alg:recursion} can be derived.
We take \(p=3\) as an example to demonstrate this recursion scheme.
Consider the generalized Sylvester equation
\begin{equation}
\label{eq:3rd-order}
S^{2}Y+SY\tilde{S}+Y\tilde{S}^{2}=R,
\end{equation}
where \(S\) and \(\tilde{S}\) are block upper triangular.
Initially \(S=\tilde{S}\), but these matrices become
different for sub-problems arising during the recursion.
Suppose that \(S\), \(\tilde{S}\), \(R\), and \(Y\)
are partitioned into~\eqref{eq:partition} such that
the diagonal blocks are of size roughly \(n/2\).
Then~\eqref{eq:3rd-order} becomes
\begin{align*}
S_{1,1}^{2}Y_{1,1}+S_{1,1}Y_{1,1}\tilde{S}_{1,1}+Y_{1,1}\tilde{S}_{1,1}^{2}
={}&R_{1,1}-(S_{1,1}S_{1,2}+S_{1,2}S_{2,2})Y_{2,1}
-S_{1,2}Y_{2,1}\cdot\tilde{S}_{1,1},\\
S_{1,1}^{2}Y_{1,2}+S_{1,1}Y_{1,2}\tilde{S}_{2,2}+Y_{1,2}\tilde{S}_{2,2}^{2}
={}&R_{1,2}-(S_{1,1}S_{1,2}+S_{1,2}S_{2,2})Y_{2,2}
-S_{1,2}Y_{2,2}\tilde{S}_{2,2}\\
&-Y_{1,1}(\tilde{S}_{1,1}\tilde{S}_{1,2}+\tilde{S}_{1,2}\tilde{S}_{2,2})
-(S_{1,1}Y_{1,1}+S_{1,2}Y_{2,1})\tilde{S}_{1,2},\\
S_{2,2}^{2}Y_{2,1}+S_{2,2}Y_{2,1}\tilde{S}_{1,1}+Y_{2,1}\tilde{S}_{1,1}^{2}
={}&R_{2,1},\\
S_{2,2}^{2}Y_{2,2}+S_{2,2}Y_{2,2}\tilde{S}_{2,2}+Y_{2,2}\tilde{S}_{2,2}^{2}
={}&R_{2,2}-Y_{2,1}(\tilde{S}_{1,1}\tilde{S}_{1,2}
+\tilde{S}_{1,2}\tilde{S}_{2,2})-S_{2,2}Y_{2,1}\tilde{S}_{1,2}.
\end{align*}

This reduces the original equation to four smaller generalized
triangular Sylvester equations, which are solved recursively
until the matrix sizes are sufficiently small and a variant of the standard
Bartels--Stewart method~\cite[Section 7]{Simoncini2016} is used.
When \(p\) grows larger, the mathematical presentation
of the block recursion becomes more complicated and
the computational cost is expected to grow significantly.
\section{Numerical experiments}
\label{sec:experiments}
In this section, we conduct some numerical experiments of
our mixed precision algorithms for computing matrix square roots.
In our tests, we employ IEEE single precision as the lower precision
and IEEE double precision as the working precision.

The test matrices are generated in the form in Table~\ref{tab:test-matrices}.
Note that test matrix \(A\) is shifted so that its real
eigenvalues are strictly positive and safe from \(0\).
In test set VII, \(H_{n}\) denotes the Hilbert matrix.
Algorithm~\ref{alg:main} is implemented for the nonsymmetric matrices,
and Algorithm~\ref{alg:spd} is reserved for the symmetric matrices.

\begin{table}[!tb]\centering
\caption{Summary of test matrices.}
\label{tab:test-matrices}
\begin{tabular}{cccc}\hline
Test set  &  Matrix type  &  Generation  &  Matrix size \(n\) \\\hline
 I   &  Nonsymmetric  &  \(A_{i,j}\sim U(0,1)\), i.i.d. plus shift  &  2048 \\
II   &  Nonsymmetric  &  \(A_{i,j}\sim N(0,1)\), i.i.d. plus shift  &  2048 \\
III  &  Nonsymmetric  &  \(A_{i,j}\sim U(0,1)\), i.i.d. plus shift  &  4096 \\
IV   &  Nonsymmetric  &  \(A_{i,j}\sim N(0,1)\), i.i.d. plus shift  &  4096 \\
 V   &    Symmetric   &
\(A_{i,j}\sim U(0,1)\), \(i\geq j\), i.i.d. plus shift &  4096 \\
VI   &    Symmetric   &
\(A_{i,j}\sim N(0,1)\), \(i\geq j\), i.i.d. plus shift &  4096 \\
VII  &    Symmetric   &  \(A=H_{n}+\tau I_{n}\), \(\tau=10^{-6}\)   &  1024 \\
\hline
\end{tabular}
\end{table}

All experiments are performed on a Linux server equipped with two sixteen-core
Intel Xeon Gold 6226R 2.90 GHz CPUs with 1024 GB of main memory.
The algorithms are implemented in LAPACK (version 3.11.0)
linked with OpenBLAS library (version 0.3.26).
When employing a lower precision, it is often necessary to apply
appropriate scaling to prevent overflow during the conversion of
data from a higher precision to a lower precision in practice.
The block size in Algorithm~\ref{alg:main} is chosen as \(\mathtt{blks}=32\).
The convergence threshold is chosen as \(\mathtt{tol}=10^{-12}\).

\subsection{Accuracy tests}
\label{subsec:accuracy-tests}
We measure two terms, one relative residual and one relative error:
\[
\epsilon_{\mathtt{residual}}=\frac{\normfro{\hat{X}^2-A}}{\normfro{A}},\qquad
\epsilon_{\mathtt{error}}=\frac{\normfro{\hat{X}-A^{1/2}}}{\normfro{A^{1/2}}}.
\]
The output of the double precision Schur algorithm
is regarded as the ground truth of \(A^{1/2}\).

In accuracy tests, we nullify the stopping criterion in order to observe
the relative residuals and errors in each iterative refinement step.
The measured terms are shown in Figure~\ref{fig:accuracy} in the \(y\)-axis,
while the \(x\)-axis displays the number of iterative refinement loops.

\begin{figure}[!tb]
\centering
\subfigure{\includegraphics[width=0.45\textwidth]{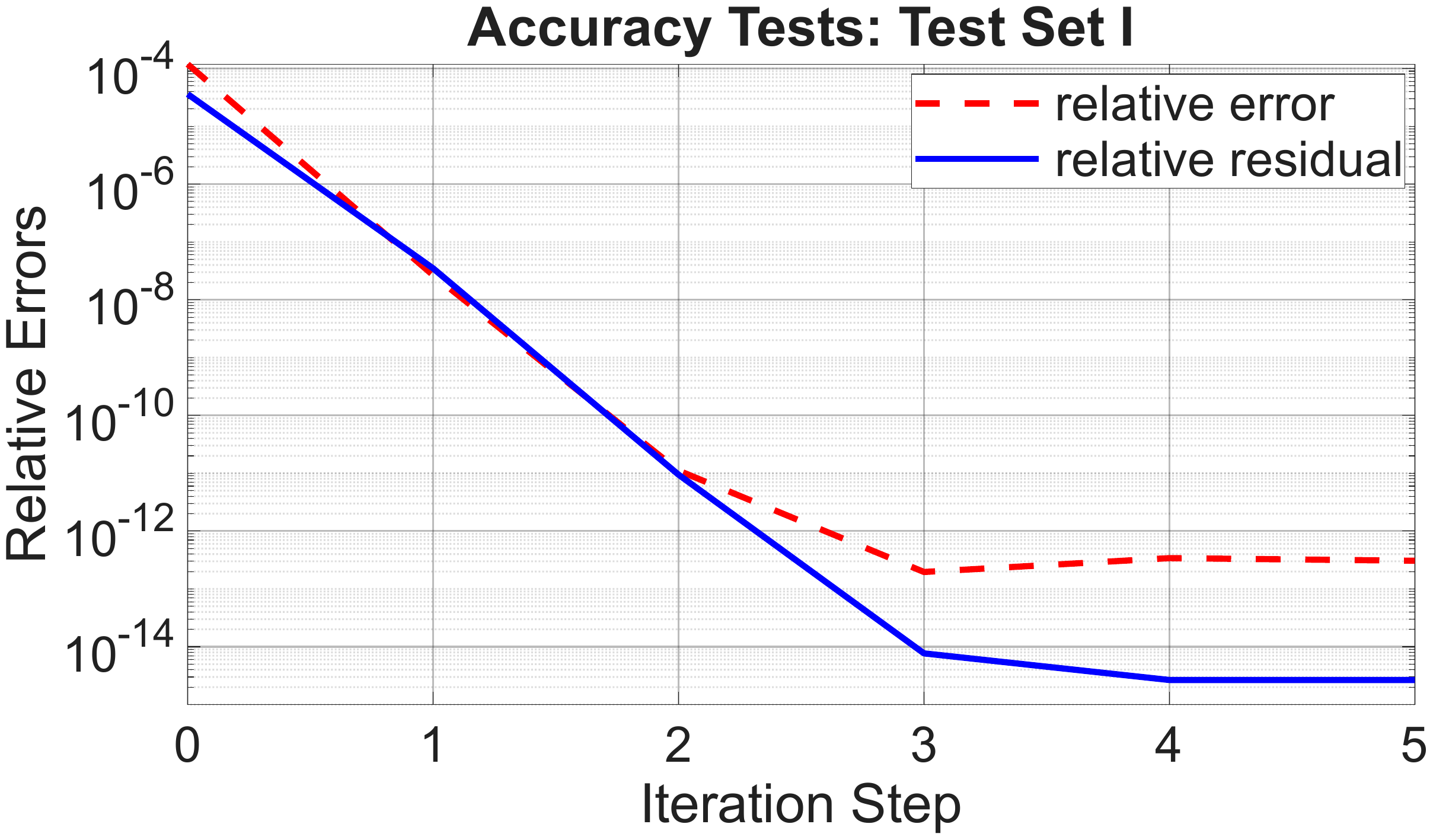}}
\subfigure{\includegraphics[width=0.45\textwidth]{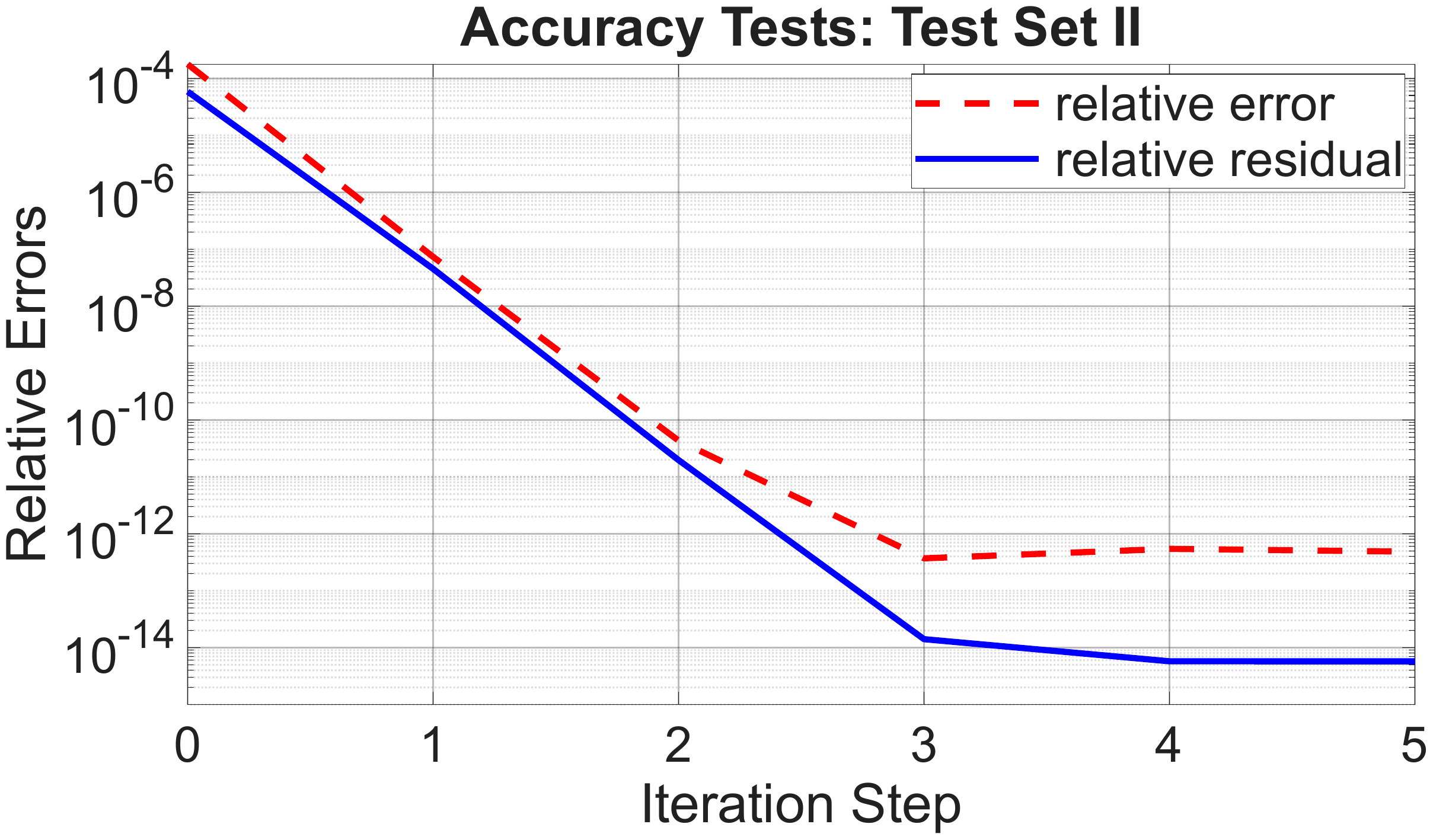}}
\subfigure{\includegraphics[width=0.45\textwidth]{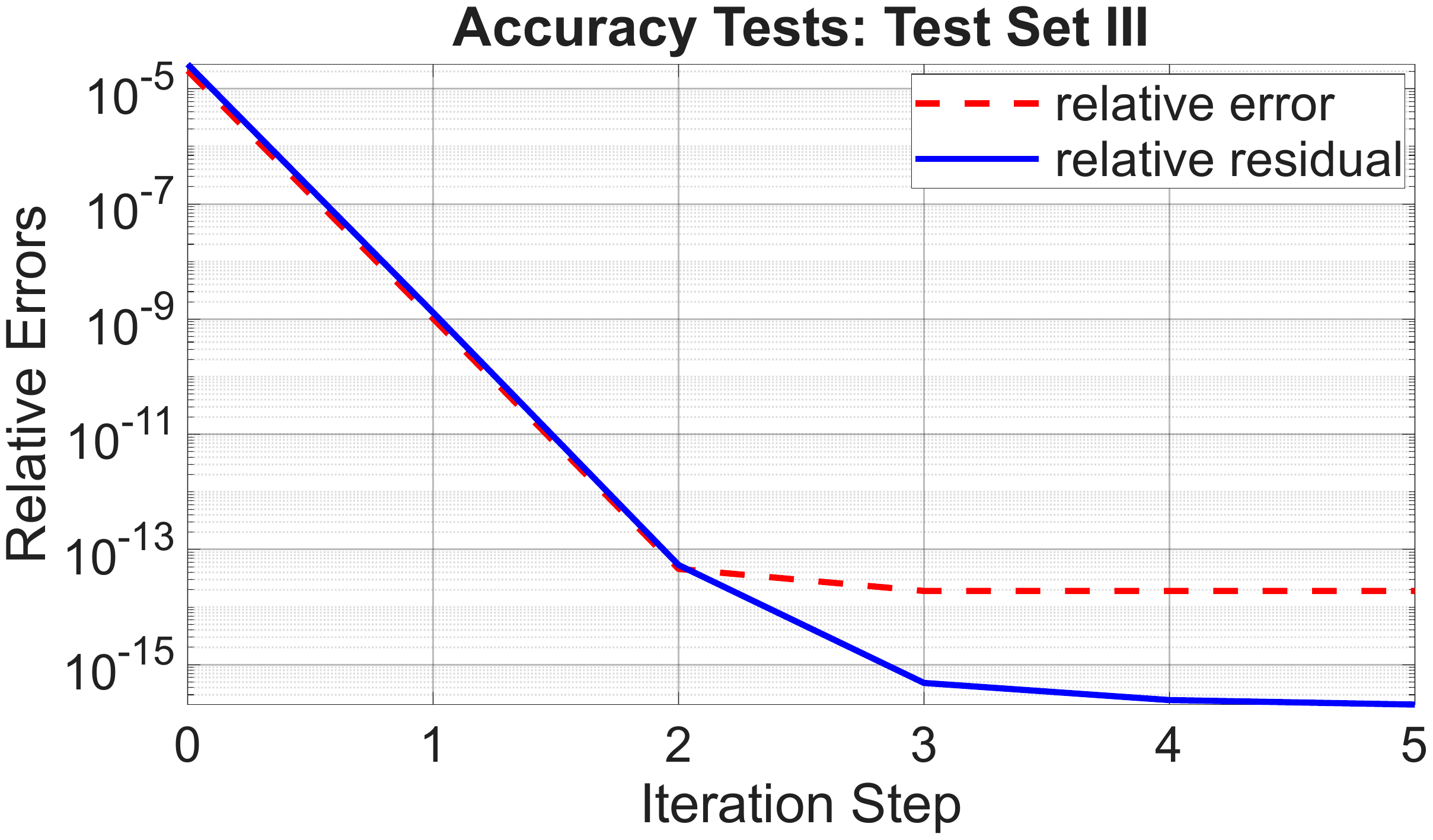}}
\subfigure{\includegraphics[width=0.45\textwidth]{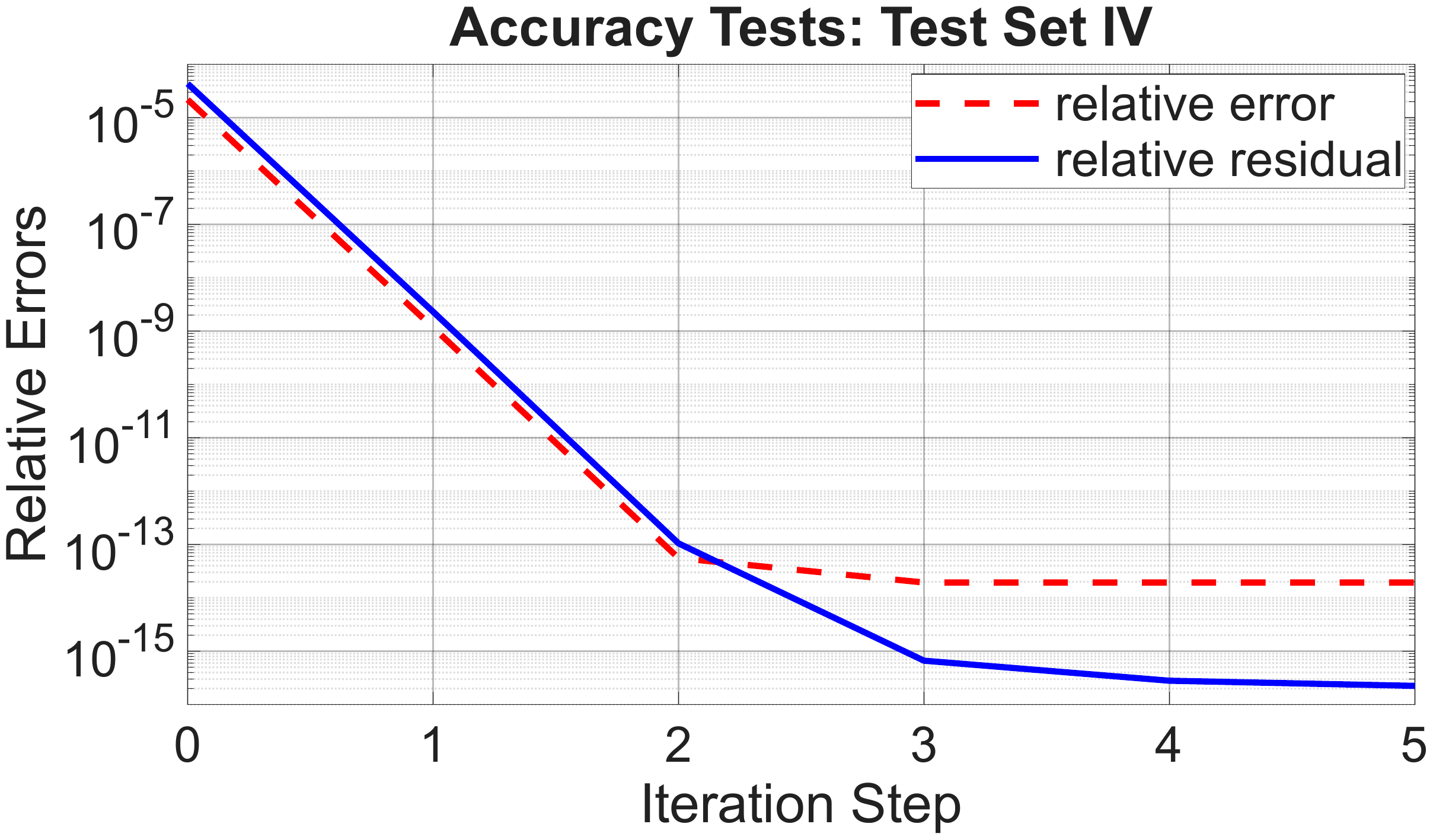}}
\subfigure{\includegraphics[width=0.45\textwidth]{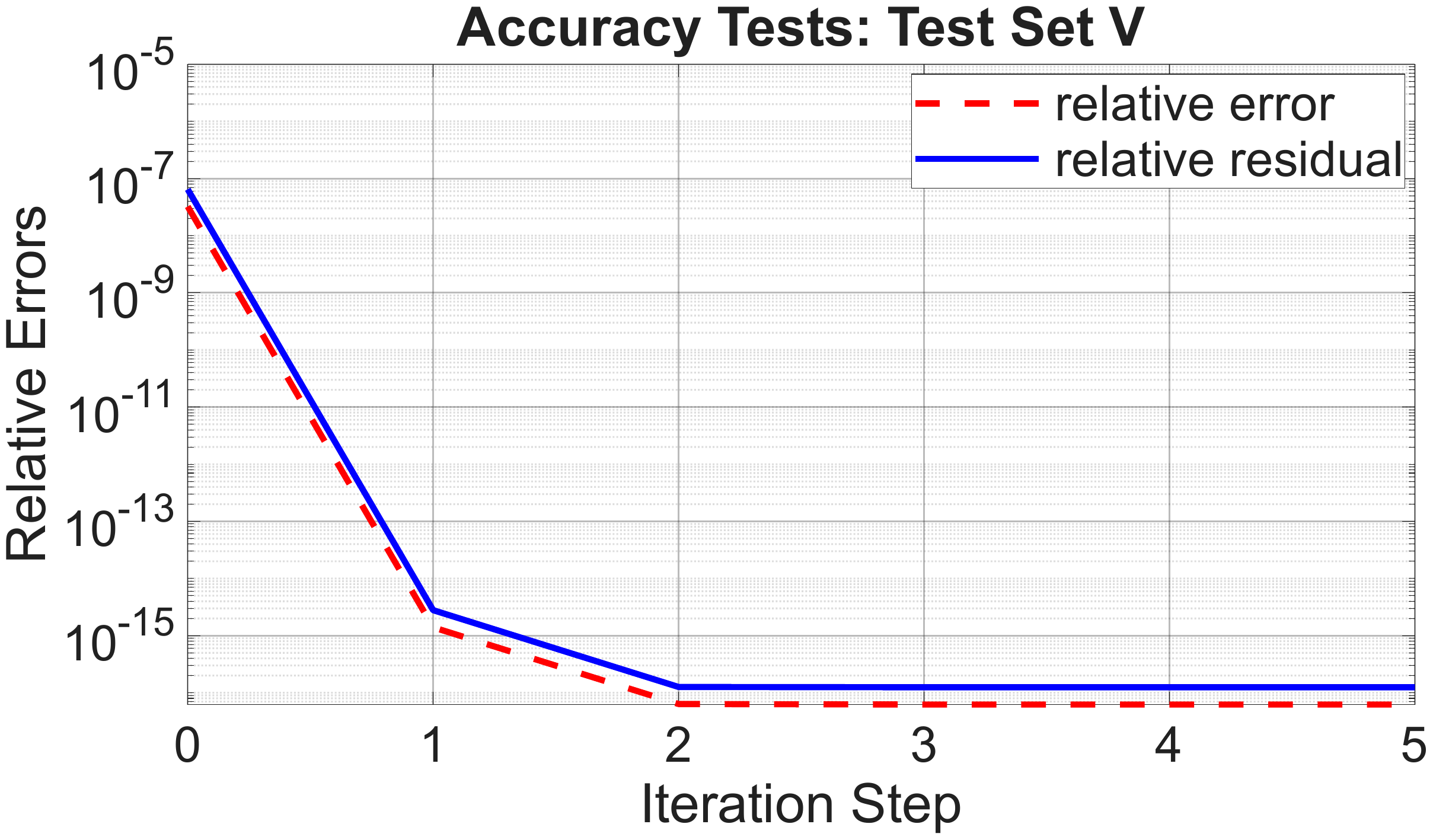}}
\subfigure{\includegraphics[width=0.45\textwidth]{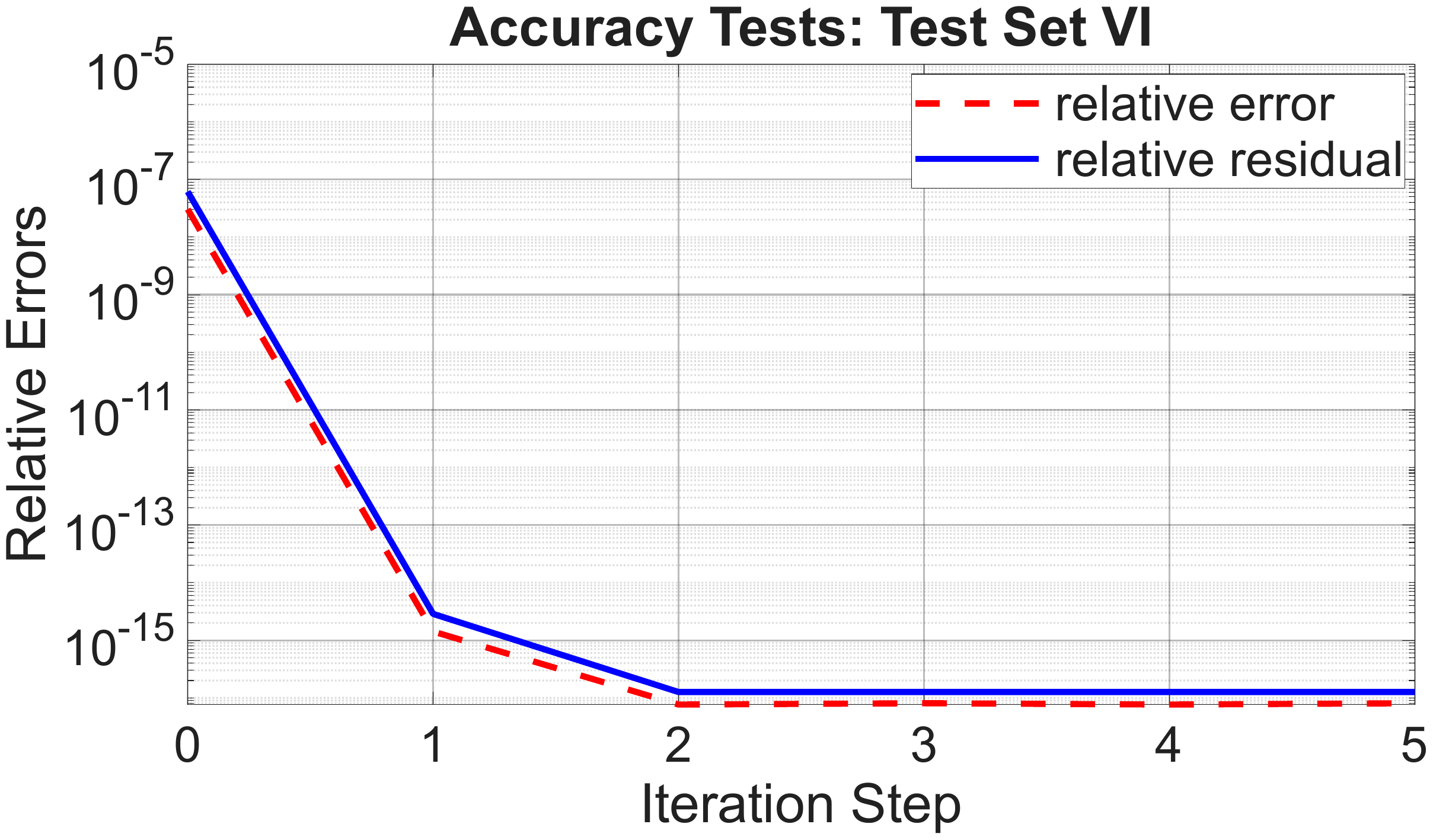}}
\subfigure{\includegraphics[width=0.45\textwidth]{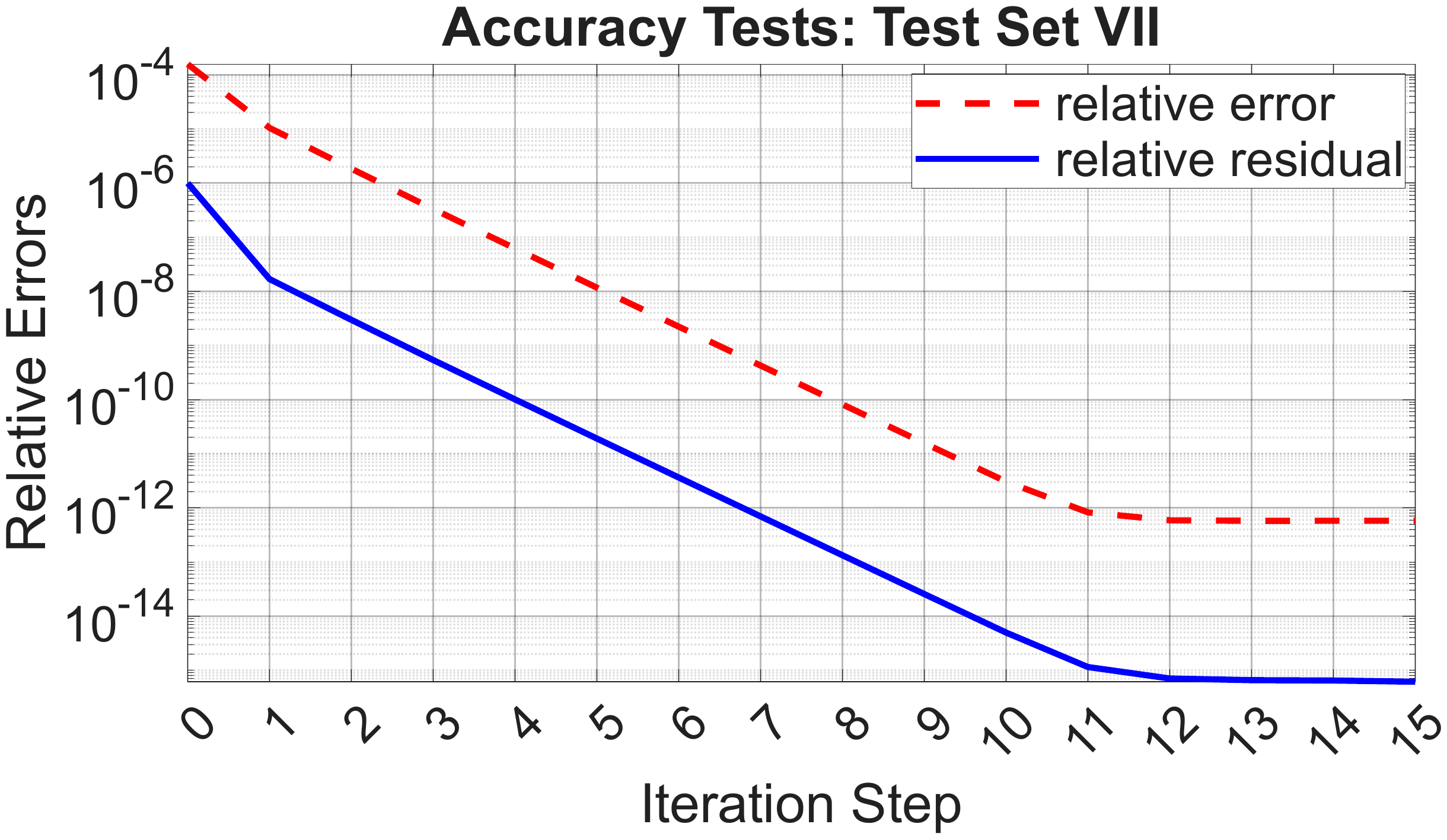}}
\caption{Relative residuals and errors of Algorithms~\ref{alg:main}
and~\ref{alg:spd} for computing matrix square roots in mixed precision,
with single precision employed as lower precision
and double precision employed as working precision.
The relative residual is displayed in the solid line,
and the relative error is displayed in the dashed line.}
\label{fig:accuracy}
\end{figure}

From Figure~\ref{fig:accuracy},
we see that if the matrix is not too ill-conditioned, it only takes
1--3 steps of iterative refinement to reach double precision accuracy.
For the highly ill-conditioned Hilbert matrix with
a regularization term, our mixed precision algorithm still converges,
although it takes around 11--12 steps of iterative refinement.
Furthermore, if the regularization factor \(\tau\)
is smaller than \(10^{-7}\), it diverges in our tests.

\subsection{Performance tests}
\label{subsec:performance-tests}
We further test the algorithm performance for the seven sets of matrices.
The results are shown in Figures~\ref{fig:performance-1}
and~\ref{fig:performance-2}.

In each plot, we record the relative execution time, i.e.,
the ratio of the wall clock time of the mixed precision
solver over that of the double precision Schur algorithm.
In Figure~\ref{fig:performance-1}, we employ labels `single', `residual',
`correction system', and `updates \& others' to represent the different
phases in Algorithms~\ref{alg:main} and~\ref{alg:spd}, respectively:
computing the initial approximation by the Schur algorithm in single precision,
computing the residual in iterative refinement, solving the correction
system in iterative refinement, and adding the correction matrix in
iterative refinement along with other \(\mathcal{O}(n^{2})\) operations.

In Figure~\ref{fig:performance-2}, we summarize the execution time by the
types of matrix operations in Algorithms~\ref{alg:main} and~\ref{alg:spd}.
We use labels `Schur decomposition', `computing triangular roots',
`trmm \& gemm', `rtrsyl', and `other \(\mathcal{O}(n^{2})\) operations'
to represent different matrix operations in our algorithms.
The results show that the Schur decomposition costs over \(20\)--\(30\)
times more than a matrix--matrix multiplication in the same precision,
which implies that it is indeed the primary cost and performance impact factor.

For most test matrices, our mixed precision algorithms offer around
\(20\%\)--\(30\%\) time savings compared to the fixed (double) precision Schur
algorithm, which constitutes a speedup around \(1.2\times\)--\(1.43\times\).
For Hilbert matrix with a regularization of \(10^{-6}\),
the iterative refinement process is very expensive,
and our mixed precision algorithm does not offer any performance gain.

\begin{figure}[!tb]
\centering
\includegraphics[scale=0.25]{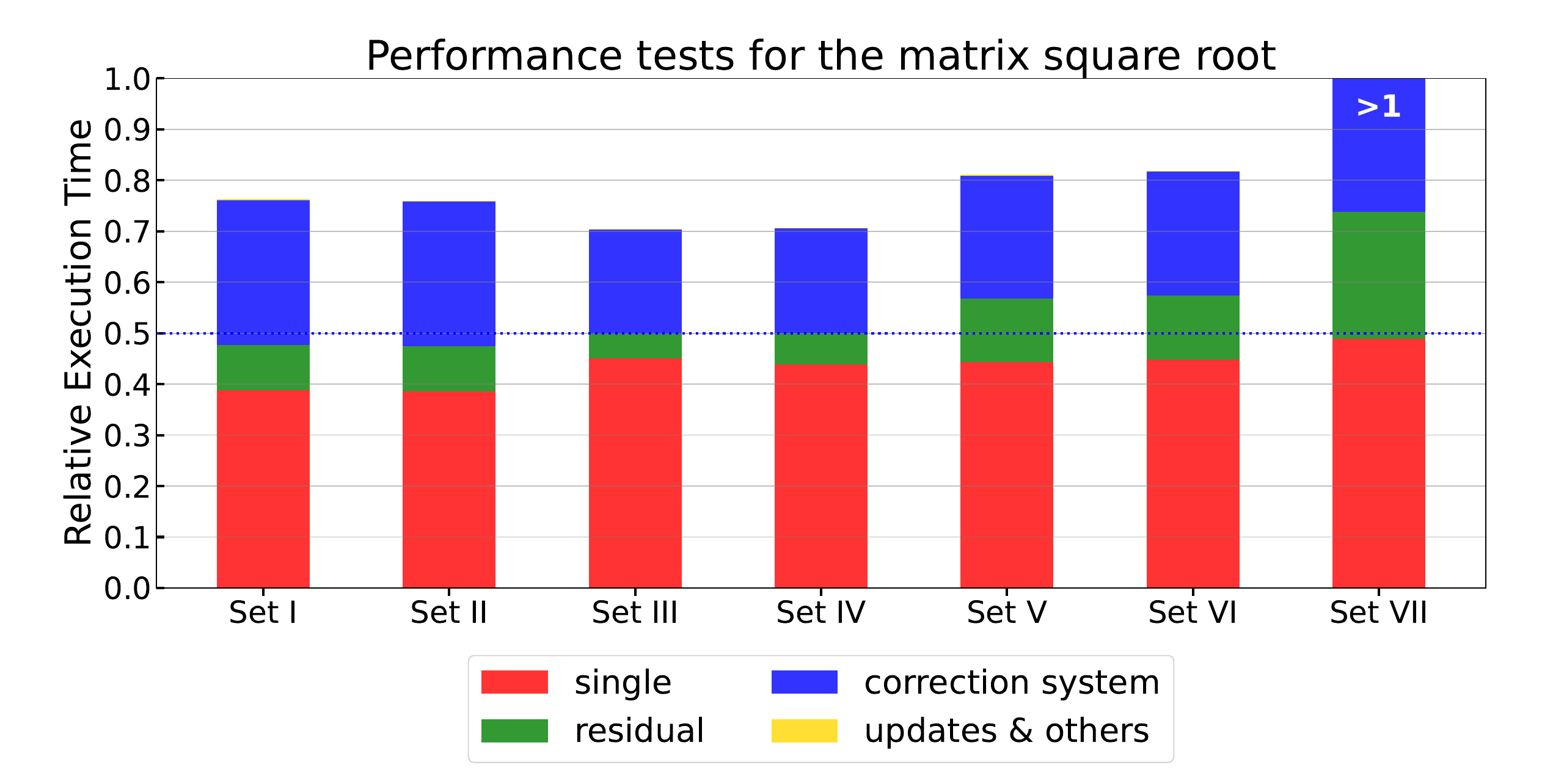}
\caption{Relative execution times of Algorithms~\ref{alg:main}
and~\ref{alg:spd} compared to the fixed (double) precision
Schur algorithm for computing matrix square roots,
with single precision employed as lower precision
and double precision employed as working precision.}
\label{fig:performance-1}
\end{figure}

\begin{figure}[!tb]
\centering
\includegraphics[scale=0.25]{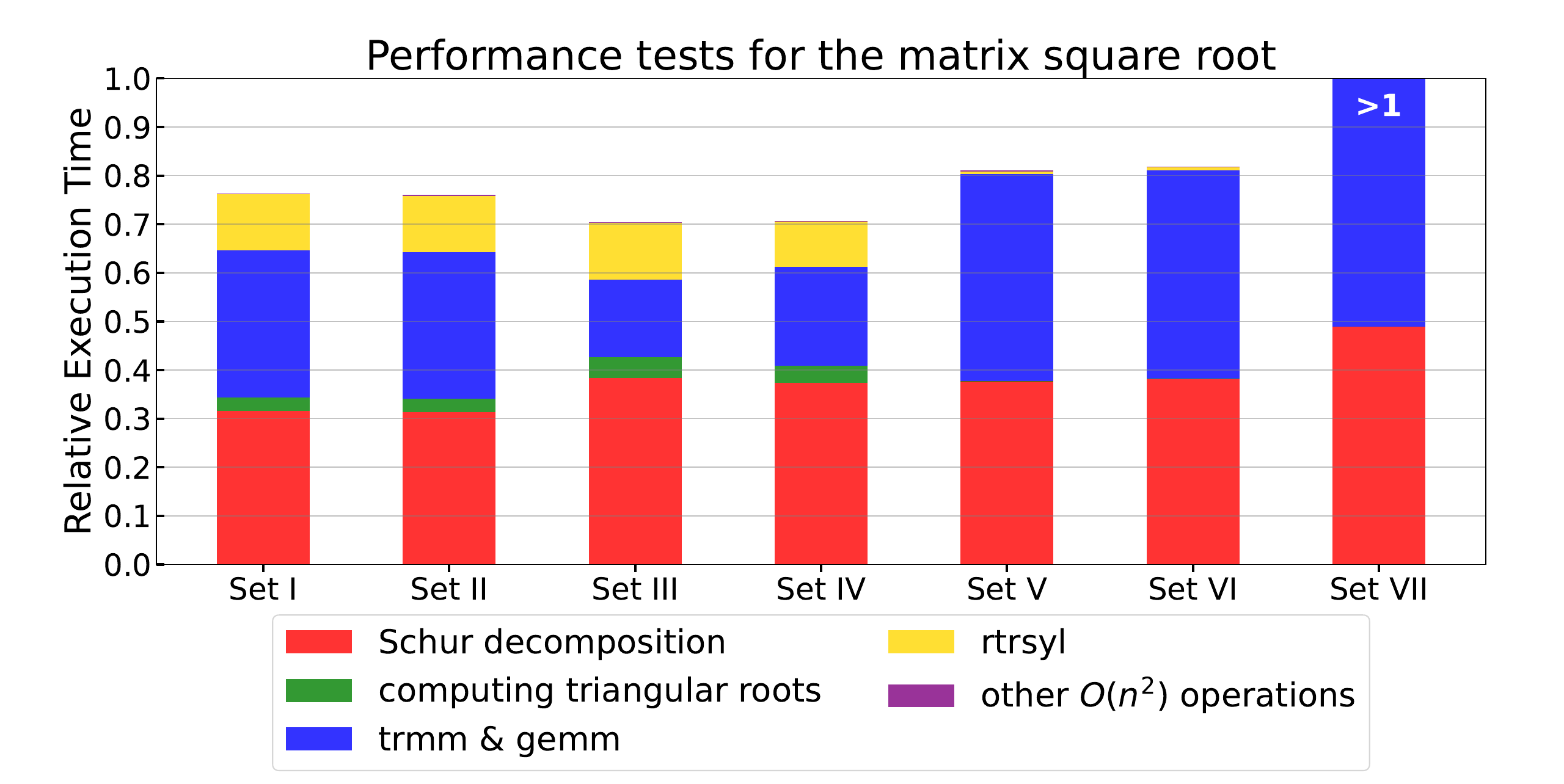}
\caption{Relative execution times of Algorithms~\ref{alg:main}
and~\ref{alg:spd} compared to the fixed (double) precision
Schur algorithm for computing matrix square roots,
with single precision employed as lower precision
and double precision employed as working precision.}
\label{fig:performance-2}
\end{figure}
\section{Conclusions}
\label{sec:conclusion}
In this paper, we propose a novel mixed precision
algorithm for computing matrix square roots.
We derive a framework that combines Schur decomposition
with iterative refinement to elevate a lower precision
approximation to the working precision level.
To enhance performance, we incorporate strategies such as
reusing Schur factors to obtain an approximate Newton
iteration and applying block recursion to solve Sylvester equations.
Based on these discussions, we present mixed precision algorithms for
computing square roots of both general and symmetric/Hermitian positive
definite matrices, as well as an extension to matrix \(p\)th roots.

Convergence analysis demonstrates that our mixed precision algorithms converge
at least linearly to the matrix square root when using exact arithmetic.
When taking rounding errors into account, we establish that our algorithm
for symmetric or Hermitian positive definite matrices recovers an error
on the order of the working precision unit-roundoff under mild conditions.
We also discuss the generalization of this result to general matrices.
Preliminary numerical experiments indicate that our algorithm frequently
reduces execution time at a rate of approximately 20--30\% compared to
the fixed (working) precision Schur algorithm on x86-64 architectures.

\section*{Acknowledgement}
The authors thank Yuji Nakatsukasa and Nian Shao for helpful discussions.

\bibliographystyle{siamplain}

\begin{thebibliography}{10}

\bibitem{Survey2021}
{\sc A.~Abdelfattah, H.~Anzt, et~al.}, {\em A survey of numerical linear
  algebra methods utilizing mixed-precision arithmetic}, Int.\ J. High
  Perform.\ Comput.\ Appl., 35 (2021), pp.~344--369,
  \url{https://doi.org/10.1177/10943420211003313}.

\bibitem{BS1972}
{\sc R.~H. Bartels and G.~W. Stewart}, {\em Solution of the matrix equation
  \({AX+XB=C}\)}, Comm. ACM, 15 (1972), pp.~820--826,
  \url{https://doi.org/10.1145/361573.361582}.

\bibitem{BHM2005}
{\sc D.~A. Bini, N.~J. Higham, and B.~Meini}, {\em Algorithms for the matrix
  \(p\)th root}, Numer.\ Algorithms, 39 (2005), pp.~349--378,
  \url{https://doi.org/10.1007/s11075-004-6709-8}.

\bibitem{BKS2023}
{\sc Z.~Bujanovi\'{c}, D.~Kressner, and C.~Schr\"{o}der}, {\em Iterative
  refinement of {Schur} decompositions}, Numer.\ Algorithms, 92 (2023),
  pp.~247--267, \url{https://doi.org/10.1007/s11075-022-01327-6}.

\bibitem{BHMV2025}
{\sc A.~Buttari, N.~J. Higham, T.~Mary, and B.~Vieubl\'{e}}, {\em A modular
  framework for the backward error analysis of {GMRES}}, IMA J. Numer.\ Anal.,
  (2025), p.~draf049, \url{https://doi.org/10.1093/imanum/draf049}.

\bibitem{CH2017}
{\sc E.~Carson and N.~J. Higham}, {\em A new analysis of iterative refinement
  and its application to accurate solution of ill-conditioned sparse linear
  systems}, SIAM J. Sci.\ Comput., 39 (2017), pp.~A2834--A2856,
  \url{https://doi.org/10.1137/17M1122918}.

\bibitem{CH2018}
{\sc E.~Carson and N.~J. Higham}, {\em Accelerating the solution of linear
  systems by iterative refinement in three precisions}, SIAM J. Sci.\ Comput.,
  40 (2018), pp.~A817--A847, \url{https://doi.org/10.1137/17M1140819}.

\bibitem{CHP2020}
{\sc E.~Carson, N.~J. Higham, and S.~Pranesh}, {\em Three-precision
  {GMRES}-based iterative refinement for least squares problems}, SIAM J. Sci.\
  Comput., 42 (2020), pp.~A4063--A4083,
  \url{https://doi.org/10.1137/20M1316822}.

\bibitem{DHR2013}
{\sc E.~Deadman, N.~J. Higham, and R.~Ralha}, {\em Blocked {Schur} algorithms
  for computing the matrix square root}, in Applied Parallel and Scientific
  Computing, Berlin, Germany, 2013, Springer, pp.~171--182,
  \url{https://doi.org/10.1007/978-3-642-36803-5_12}.

\bibitem{Fasi2019}
{\sc M.~Fasi}, {\em Computing Matrix Functions in Arbitrary Precision
  Arithmetic}, PhD thesis, University of Manchester, 2019.

\bibitem{GMS2025-2}
{\sc B.~Gao, Y.~Ma, and M.~Shao}, {\em Mixed precision iterative refinement for
  least squares with linear equality constraints and generalized least squares
  problems}, Numer.\ Linear Algebra Appl., 32 (2025), p.~e70036,
  \url{https://doi.org/10.1002/nla.70036}.

\bibitem{GMS2025-1}
{\sc W.~Gao, Y.~Ma, and M.~Shao}, {\em A mixed precision {Jacobi} {SVD}
  algorithm}, ACM Trans.\ Math.\ Software, 51 (2025),
  \url{https://doi.org/10.1145/3721124}.

\bibitem{GV2013}
{\sc G.~H. Golub and C.~F. Van~Loan}, {\em Matrix Computations}, Johns Hopkins
  University Press, Baltimore, MD, USA, 4th~ed., 2013.
  \url{https://epubs.siam.org/doi/book/10.1137/1.9781421407944}.

\bibitem{Higham2002}
{\sc N.~J. Higham}, {\em Accuracy and Stability of Numerical Algorithms}, SIAM,
  Philadelphia, PA, USA, 2nd~ed., 2002,
  \url{https://doi.org/10.1137/1.9780898718027}.

\bibitem{Higham2008}
{\sc N.~J. Higham}, {\em Functions of Matrices}, SIAM, Philadelphia, PA, USA,
  2008, \url{https://doi.org/10.1137/1.9780898717778}.

\bibitem{HL2021}
{\sc N.~J. Higham and X.~Liu}, {\em A multiprecision derivative-free
  {Schur}--{Parlett} algorithm for computing matrix functions}, SIAM J. Matrix
  Anal.\ Appl., 42 (2021), pp.~1401--1422,
  \url{https://doi.org/10.1137/20M1365326}.

\bibitem{JK2002-1}
{\sc I.~Jonsson and B.~K{\aa}gstr{\"o}m}, {\em Recursive blocked algorithms for
  solving triangular systems---{Part I}: one-sided and coupled {Sylvester}-type
  matrix equations}, ACM Trans.\ Math.\ Software, 28 (2002), pp.~392--415,
  \url{https://doi.org/10.1145/592843.592845}.

\bibitem{JK2002-2}
{\sc I.~Jonsson and B.~K{\aa}gstr{\"o}m}, {\em Recursive blocked algorithms for
  solving triangular systems---{Part II}: two-sided and generalized {Sylvester}
  and {Lyapunov} matrix equations}, ACM Trans.\ Math.\ Software, 28 (2002),
  pp.~416--435, \url{https://doi.org/10.1145/592846.592847}.

\bibitem{JK2003}
{\sc I.~Jonsson and B.~K{\aa}gstr\"{o}m}, {\em {RECSY}---{A} high performance
  library for {Sylvester}-type matrix equations}, in Euro-Par 2003 Parallel
  Processing, Cham, Switzerland, 2003, Springer, pp.~810--819,
  \url{https://doi.org/10.1007/978-3-540-45209-6_111}.

\bibitem{KMS2023}
{\sc D.~Kresnner, Y.~Ma, and M.~Shao}, {\em A mixed precision {LOBPCG}
  algorithm}, Numer.\ Algorithms, 94 (2023), pp.~1653--1671,
  \url{https://doi.org/10.1007/s11075-023-01550-9}.

\bibitem{LLLKBD2006}
{\sc J.~Langou, J.~Langou, P.~Luszczek, J.~Kurzak, A.~Buttari, and
  J.~Dongarra}, {\em Exploiting the performance of 32 bit floating point
  arithmetic in obtaining 64 bit accuracy (revisiting iterative refinement for
  linear systems)}, in Proceedings of the 2006 ACM/IEEE Conference on
  Supercomputing, SC '06, 2006, pp.~113--es,
  \url{https://doi.org/10.1145/1188455.1188573}.

\bibitem{Liu2022}
{\sc X.~Liu}, {\em Computing Matrix Functions in Arbitrary Precision
  Arithmetic}, PhD thesis, \newline University of Manchester, 2022.

\bibitem{Liu2025}
{\sc X.~Liu}, {\em Mixed-precision {Paterson}--{Stockmeyer} method for
  evaluating polynomials of matrices}, SIAM J. Matrix Anal.\ Appl., 46 (2025),
  pp.~811--835, \url{https://doi.org/10.1137/24M1675734}.

\bibitem{Moler1967}
{\sc C.~B. Moler}, {\em Iterative refinement in floating point}, J. ACM, 14
  (1967), pp.~316--321, \url{https://doi.org/10.1145/321386.321394}.

\bibitem{OA2018}
{\sc T.~Ogita and K.~Aishima}, {\em Iterative refinement for symmetric
  eigenvalue decomposition}, Jpn.\ J. Indust.\ Appl.\ Math., 35 (2018),
  pp.~1007--1035, \url{https://doi.org/10.1007/s13160-018-0310-3}.

\bibitem{OA2019}
{\sc T.~Ogita and K.~Aishima}, {\em Iterative refinement for symmetric
  eigenvalue decomposition {II}: clustered eigenvalues}, Jpn.\ J. Indust.\
  Appl.\ Math., 36 (2019), pp.~435--459,
  \url{https://doi.org/10.1007/s13160-019-00348-4}.

\bibitem{OA2020}
{\sc T.~Ogita and K.~Aishima}, {\em Iterative refinement for singular value
  decomposition based on matrix multiplication}, J. Comput.\ Appl.\ Math., 369
  (2020), p.~112512, \url{https://doi.org/10.1016/j.cam.2019.112512}.

\bibitem{OC2024}
{\sc E.~Oktay and E.~Carson}, {\em Mixed precision {Rayleigh} quotient
  iteration for total least squares problems}, Numer.\ Algorithms, 96 (2024),
  pp.~777--798, \url{https://doi.org/10.1007/s11075-023-01665-z}.

\bibitem{OR2000}
{\sc J.~M. Ortega and W.~C. Rheinboldt}, {\em Iterative Solution of Nonlinear
  Equations in Several Variables}, SIAM, Philadelphia, PA, USA, 2000,
  \url{https://doi.org/10.1137/1.9780898719468}.

\bibitem{Simoncini2016}
{\sc V.~Simoncini}, {\em Computational methods for linear matrix equations},
  SIAM Rev., 58 (2016), pp.~377--441, \url{https://doi.org/10.1137/130912839}.

\end{thebibliography}

\end{document}